\documentclass[12pt, a4paper, reqno]{amsart}

\usepackage{epsfig}
\usepackage{graphicx}
\usepackage{setspace}
\usepackage{mathrsfs}
\usepackage{cancel}
\usepackage{faktor}
\usepackage{mhequ}

\usepackage{stmaryrd}
\usepackage{amssymb}
\usepackage{version}
\usepackage[T1]{fontenc}
\usepackage{amsmath}
\usepackage{amsthm}

\newcommand{\cqfd}{\; \;\underset{\Box}{}}
\newcommand{\R}{\textrm{I\hspace{-.19em}R}}

\newcommand{\w}{\widetilde}
\renewcommand{\t}{\tilde}
\newcommand{\wg}{\tilde{g}}
\newcommand{\nw}{\tilde{\nabla}}
\newcommand{\contrainte}{\mathcal{C}}
\renewcommand{\d}{\partial}

\newcommand{\Bz}{\mathring{B}}
\newcommand{\Ps}{P^{*}}
\newcommand{\Pgpi}{P_{(g,\pi)}}
\newcommand{\Pgpis}{P_{(g,\pi)}^{*}}
\newcommand{\Psw}{\widetilde{P}^{*}}
\newcommand{\Pgpisw}{P_{(\t{g},\t{\pi})}^{*}}
\newcommand{\Sz}{\mathring{S}}
\newcommand{\Uz}{\mathring{U}}
\newcommand{\phibf}{\mathbf{\Phi}}
\newcommand{\phibfs}{\mathbf{\Phi}^{*}}
\newcommand{\phisun}{\mathbf{\Phi}^{*}_{1}}
\newcommand{\phisdeux}{\mathbf{\Phi}^{*}_{2}}
\newcommand{\phiwsun}{\widetilde{\mathbf{\Phi}}^{*}_{1}}
\newcommand{\phiwsdeux}{\widetilde{\mathbf{\Phi}}^{*}_{2}}
\newcommand{\phio}{\mathbf{\Phi_{0}}}

\newcommand{\phiogpi}{\mathbf{\Phi_{0}}(g, \pi)}

\newcommand{\phigpi}{\mathbf{\Phi}(g, \pi)}
\newcommand{\phigpiws}{\mathbf{\Phi}(\t{g},\t{\pi})^{*}}
\newcommand{\phiogs}{\mathbf{\Phi_{0}}(g)^{*}}

\newcommand{\phiogpis}{\mathbf{\Phi_{0}}(g,\pi)^{*}}
\newcommand{\phiigpis}{\mathbf{\Phi_{i}}(g,\pi)^{*}}

\newcommand{\phigpis}{\mathbf{\Phi}(g,\pi)^{*}}
\newcommand{\phii}{\mathbf{\Phi_{i}}}

\newcommand{\phiigpi}{\mathbf{\Phi_{i}}(g,\pi)}
\newcommand{\nat}{\mathbb{N}}
\newcommand{\G}{\mathcal{G}}
\newcommand{\Gplus}{\mathcal{G}^{+}}
\renewcommand{\L}{\mathcal{L}}
\renewcommand{\S}{\mathcal{S}}
\newcommand{\T}{\mathcal{T}}
\newcommand{\K}{\mathcal{K}}
\newcommand{\F}{\mathcal{F}}

\newcommand{\precig}[1]{\hspace{0.05cm}^{(#1)}\!}

\renewcommand{\epsilon}{\varepsilon}

\newcommand{\cde}[3]{\Gamma_{#1 \; #3}^{\hspace{0.1cm} #2}}
\newcommand{\wcde}[3]{\widetilde{\Gamma}_{#1 \; #3}^{\hspace{0.1cm} #2}}

\newcommand{\ade}[3]{A_{#1 \; #3}^{\hspace{0.1cm} #2}}
\newcommand{\wade}[3]{\widetilde{A}_{#1 \; #3}^{\hspace{0.1cm} #2}}

\newcommand{\cdez}[3]{\mathring{\Gamma}_{#1 \; #3}^{\hspace{0.1cm} #2}}

\newcommand{\infeg}{\leqslant}
\newcommand{\A}{\mathcal{A}}
\newcommand{\supeg}{\geqslant}

\newcommand{\cinf}{\mathscr{C}^{\infty}}

\newcommand{\demi}{\frac{1}{2}}
\newcommand{\tdemi}{\tfrac{1}{2}}

\newcommand{\Norm}[2]{||{#1}||_{#2}}

\newcommand{\bNorm}[2]{\Big\| {#1} \Big\|_{#2}}
\newcommand{\BNorm}[2]{\bigg\| {#1} \bigg\|_{#2}}
\newcommand{\norm}[2]{|{#1}|_{#2}}

\newcommand{\leb}[2]{{L}^{#1}_{#2}}

\newcommand{\sob}[3]{W^{#1,#2}_{#3}}

\newcommand{\m}{\mathcal{M}}
\newcommand{\N}{\mathcal{N}}

\newcommand{\ie}{\emph{i.e.} \,}
\newcommand{\gz}{\mathring{g}}

\newcommand{\Kz}{\mathring{K}}
\newcommand{\piz}{\mathring{\pi}}
\newcommand{\n}{\nabla}
\newcommand{\nz}{\mathring{\nabla}}
\newcommand{\Tz}{\mathring{T}}
\newcommand{\wn}{\widetilde{\nabla}}
\newcommand{\wlapla}{\widetilde{\Delta}}
\newcommand{\laplaz}{\mathring{\Delta}}

\newcommand{\laplag}{\Delta_{g}}

\newcommand{\quotient}[2]{#1_{\! \diagup_{\! #2}}}

\newcommand{\Riem}{\operatorname{Riem}}
\newcommand{\Ric}{\operatorname{Ric}}
\renewcommand{\div}{\operatorname{div}}
\newcommand{\tr}{\operatorname{tr}}
\newcommand{\Coker}{\operatorname{Coker}}
\renewcommand{\Im}{\operatorname{Im}}

\newtheorem{lem}{Lemma}
\newtheorem{thm}{Theorem}
\newtheorem{cor}{Corollary}
\newtheorem{prop}{Proposition}
\newtheorem{defi}{Definition}
\newtheorem{remark}{Remark}

\title[On Bartnik Hilbert manifold structure ]{ Bartnik Hilbert manifold structure on   fibers of the scalar curvature and the constraint operator }

\author[E.  Delay]{Erwann
Delay} \address{Erwann Delay, Avignon Universit\'e, Laboratoire de Math\'ematiques d'Avignon (EA 2151), 301 rue Baruch de Spinoza,
F-84916 Avignon, France}
\email{Erwann.Delay@univ-avignon.fr}
\urladdr{http://www.math.univ-avignon.fr}

\date{\today}
\usepackage{xcolor}
\usepackage{cite}

\newcounter{mnotecount}[section]

\newcommand{\kk}{\kappa}
\begin{document}

\begin{abstract}
We adapt the Bartnik method to provide a Hilbert manifold structure  for the space of  solutions, without KID's, to the vacuum constraint equations on compact  manifold of any dimension $\geq 3$. 
In the course, we prove that some fibers of the scalar curvature or the constraint  operator  are Hilbert submanifolds.
We also study some operators and  inequalities related to  the KID's operator. Finally we comment the adaptation to some non compact manifolds.
\end{abstract}

\maketitle

\tableofcontents



\section{Introduction}
The linearisation stability studies of Fischer, Marsden and Moncrief, started in 1975 (see \cite{FM1975, FM1979, FMM1980 }) implies the existence of a Fr\'echet manifold structure (modelled on $C^{\infty}$) for the set of solutions of the vacuum constraint equations on a compact manifold.
A similar structure has been proven by Andersson in the asymptotically flat setting in 1987,  \cite{Andersson1987}.
With P. Chru\'sciel, we obtained a Banach manifold structure (modelled on $C^{k,\alpha}$) of such set of solutions  in three classical context of compact,  or asymptotically flat ,or asymptotically hyperbolic manifolds in 2004, \cite{CD2004}.
At the same time R. Bartnik, provided a Hilbert manifold structure (modelled on $H^2\times H^1$) on three dimensional  asymptotically flat manifolds (\cite{Bartnik2005}, see also \cite{McCormick2014},\cite{McCormick2015}, \cite{RaiSaraykar2016} for some adaptation to coupled equations in dimension 3). The  weak regularity assumptions concerning the metric involved (curvature only in $\leb{2}{}$ in dimension 3)  can be related to the context of the bounded $L^{2}$ curvature  of 
S. Klainerman, I. Rodnianski and J.
Szeftel \cite{KRS2015}.
Note also that  {R.} Bartnik showed in \cite{Bartnik1986} that these assumptions on the regularity are the weakest possible to define the ADM mass of the manifold.

With J. Fougeirol, we obtained an asymptotically hyperbolic version of the Bartnik   Hilbert structure \cite{DF2016} on 3-manifolds. In the course we had to introduce and study two natural operators of second order $T$ and $U$ in order to overcome special difficulties  to this asymptotic.

We could   also mention here the weak  {\it  local } (around a more regular solution) Banach manifold  structure modelled on $W^{2,p}\times W^{1,p}$ ($p>n$) for AF manifold,  which follows  from lemma 2.10 in \cite{HuangLee2019} (which corrected a {\it global} structure affirmation of the papers cited as [8] and [11] there, see remark 2.11 of  \cite{HuangLee2019}).

The goal of the present paper is to obtain a Hilbert manifold structure, modelled on $H^{k+2}\times H^{k+1}$, on an $n$-dimensional compact manifold for $ k +2>\frac n2$.

Thus, in the compact setting, we obtain a non trivial generalization  of the Bartnik- Hilbert structure either for the dimension but also for the regularity.

Before going to  the main theorem, let us introduce some notations. The constraint operator $\phibf$ we are studying act on couples of the form $(g, \pi)$ where $g$ is Riemannian metric and  the field $\pi$
is a contravariant symmetric two tensor field valued in the n-forms. Its thorough definition is given in section \ref{sectionopecontrainte}, but  let us write here 
$\phibf (g,\pi)$ as follows:
\begin{eqnarray*}
\phiogpi &:=& \left(R(g) - 2 \Lambda\right) \sqrt{g} - \left(\norm{\pi}{g}^{2} - \tfrac{1}{n-1} ( {\tr}_{g}\pi)^{2}\right) \slash \sqrt{g},\\
\phiigpi &:=&g_{ij} \n_{k} \pi^{jk},
\end{eqnarray*}
where $R(g)$, $\sqrt{g}$, $\nabla$,  are respectively   the scalar curvature, the volume form, and the connexion  of $g$,  and $\Lambda$ is the cosmological constant.
We start with the scalar curvature operator.
\begin{thm}{}\label{mainthmmanifoldintroscal}
Let $\m,$ be a smooth n-dimensional compact manifold with $n\geq 3$.  Let $k\in\mathbb N$ such that $k+2>n/2$.
Let $R: H_+^{k+2}\rightarrow  H^k$ be the scalar curvature operator.
Then for every $\epsilon \in H^k$ , the  non-static fiber  of the scalar curvature map:
$$\mathcal S_0(\epsilon) := \lbrace g \in  H_+^{k+2} : \;\ker DR(g)^* =\{0\},\; R(g)= \epsilon \rbrace$$
is a submanifold of $H^{k+2}$. In particular,  the space of non static solutions of the constant scalar curvature equation  has a Hilbert submanifold structure.
\end{thm}
It is easy to prove that if  $\epsilon$ is negative, then any solution is non static, but  in more general situation the non staticity may be proved (see \cite{BEM1976, Bourguignon1975, FM1975} for smooth metrics).\\

For the constraint map, we
work with a phase space $\F$ consisting of pairs $(g,\pi)$ of $H^{k+2}\times H^{k+1}$ regularity, 

We now state  a formal version of our main result (see Theorem \ref{mainthmmanifold} for a precise statement).
\begin{thm}{}\label{mainthmmanifoldintro} Let $n\geq 3$ and let $k\in\mathbb N$ such that $k+2>n/2$.
Let $\m$ be a smooth n-dimensional compact manifold.
Let $\phibf : \F \rightarrow \L^{*}$ be the constraint operator associated to the cosmological constant $\Lambda$.
For every $\epsilon \in \L^{*}$  , the no KID's  fiber  of the constraint map 
$$\contrainte_0 (\epsilon) := \lbrace (g,\pi) \in \F : \;\ker D\phigpi^* =\{0\},\; \phigpi = \epsilon \rbrace$$
is a submanifold of $\F$. In particular,  the space of solutions of the vacuum constraint equations without KID's  $\contrainte = \contrainte_0 (0)$ has a Hilbert submanifold structure.
\end{thm}

The low regularity of the metrics involved (curvature may be unbounded), and the non linear characteristic of the constraint operator, forces us to a very precise analysis of the different
 steps of the proof. All the usual instant thoughts with more regularity have to be reconsidered, including for instance, boundedness of operators, elliptic estimates or Fredholm properties (the operators also have  weak regular  coefficients here).

The kernel of the adjoint operator can be computed in a $H^k$
 space and not  in the all dual space $H^{-k}$, this regularity result take an important  part of the paper and has his own  interest.


The paper is written in the spirit of an easy adaptation to the non compact setting, where a reference metric $\gz$ with an asymptotic constant curvature on each end is usually used.
 For this future applications, we recall the definitions of a Hessian-type operator $\Tz$ and a differential operator of order two, called $\Uz$, Combined from the first  derivatives of the Killing operator $\Sz$.  As in \cite{DF2016}  they can be useful for non compact setting, once  obtained new Poincar\'e and Korn-type estimates of second order on the ends. These estimates   are very  important to prove there is no KID's (with appropriate behaviour).\\

\noindent
{\bf Acknowledgement :}  I wish to thank Piotr Chru\'sciel for his always valuable comments. I am grateful to Olaf M\"uller for a constructive exchange   about \cite{Muller2017} and  \cite{Muller2017erratum}.
The author was supported by the grant ANR-17-CE40-0034 of the French National Research Agency ANR (project CCEM).
\section{Notations and conventions}
Let $(\m,g)$ be a Riemannian manifold. We define  $T^{r}_{m}(\m)$
to be the  bundle of  tensor  covariant of  rank $m$ and contravariant of  rank $r$ .
For all $u \in T^{r}_{m}(\m)$, $\norm{u}{g}$ will denote the norm of $u$ with respect to the metric $g$ . $d \mu (g)$ is the Riemannian measure determined by $g$. $\Riem \, g , \Ric \, g$ and $R(g)$ are respectively the Riemann tensor, the Ricci tensor and the scalar curvature of the metric $g$. For a Riemannian metric $g$ with connection $\n$, we set the following notations concerning the Hessian and Laplacian of a function $u$:
\begin{eqnarray*}
\n^{2}_{ij}u &=& \n_{i}\n_{j}u,\\
\Delta u &=&  {\tr}_{g}\n^{2}u = g^{ij} \n^{2}_{ij}u.
\end{eqnarray*}

We  fix a smooth Riemannian metric $\gz$. When working on non compact manifold such as asymptotically Euclidian or asymptotically hyperbolic, it is convenient to choose for $\gz$ a special metric, for instance having Ricci curvature asymptotic to a constant at infinity like the model spaces.

We  define the following norms for the usual Lebesgue  spaces $L^p$ of functions or tensor fields.
\begin{eqnarray*}
\forall \; 1\infeg p < \infty \;\; \ \; \; \Norm{u}{p} = \left( \int_{\m} \norm{u}{\gz}^{p} \;  d\mu (\gz) \right)^{1 \slash p}\\
\text{For} \; p= \infty \; \; \; \; \Norm{u}{\infty} = \underset{\m}{sup} \left(  \norm{u}{\gz} \right)
\end{eqnarray*}

The norm on the Sobolev space $\sob{k}{p}{}$ is defined as 

$$\Norm{u}{k,p} = \sum_{|\alpha| \infeg k} \Norm{\nz^{\alpha}u}{p}$$
where $\alpha$ is a  multi-index of size $n$ and  $\nz^{\alpha}u = \nz_{i_{1}}^{\alpha_{1}} \ldots \nz_{i_{n}}^{\alpha_{n}}u$
$$\alpha = (\alpha_{1}, \ldots , \alpha_{n}) \; \; \text{ and } \; \; |\alpha| = \sum_{i = 1}^{n} \alpha_{i}$$

If needed, we specify  $\sob{k}{p}{}(T^{r}_{m} \m)$  for the corresponding Sobolev space on tensors of type $(r,m)$ on $\m$. 

For the special case where $p=2$ we could use the notation $H^k$ for $\sob{k}{2}{}$. In that context we also need
to define  $H^{k}$ for  negative integers $k$.  Let $k\in\mathbb N$, for any $v\in L^2(M,T)$, we define the linear continuous map $L_v: H^k(M,T^*\otimes\Lambda^n\m)\rightarrow  \mathbb R$ by
$$
L_v(u)=\langle v,u\rangle=\int_{\m} v\,u.
$$
$H^{-k}$ is the completion  of $L^2$ for the norm
$$
||v||_{-k}=||L_v||=\sup_{u\in H^k}\frac{\langle v,u\rangle}{||u||_{k}}.
$$
$H^{-k}$ is the dual of $H^k$.  The $L^2$-product $L_v(u)=\langle v,u\rangle$  defined for $v\in L^2$ and $u\in H^k$ can be extended  to $v\in H^{-k} $ by
$$
\langle v,u\rangle=\lim_{n\rightarrow \infty}\langle v_n,u\rangle,
$$
where $v_n\in L^2$ tends to $v$ in  $H^{-k}$.

By the Riesz representation theorem, for $v\in H^{-k}$  there exist a unique $w\in H^k$ such that
$v(u)=\langle w,u\rangle_{H^k}$. So $H^{-k}$ correspond to distributions  on $\m$ who can be written in the form
$$
v=\sum_{0\leq l\leq k\;, \;\;0\leq i_1,..., i_l\leq n} (-1)^l\nz^{i_1}....\nz^{i_l}\nz_{i_1}....\nz_{i_l}w
$$
for some  $w$  in $H^k$, or the more usual form
$$
v=\sum_{0\leq l\leq k\;, \;\;0\leq i_1,..., i_l\leq n} \nz^{i_1}....\nz^{i_l} \;^l W_{i_1,...i_l},
$$
where the $^l W$'s are covariant tensors  of rank $l$ in $L^2$. Similarly  we may define Sobolev spaces  of negative order   $W^{-k,p}$ and more generally $W^{s,p}$ for any $s\in\R$ (see definitions 2 and 3 in \cite{HNT2009}).

We also recall the usual H\"older spaces , $C^{s, \alpha} (\m,g)$ with  $0< \alpha<1$ and the following  norm
$$\Norm{u}{C^{s, \alpha}} = \underset{|k| \infeg s}{\text{max}} \Norm{\nz^{k}u}{C^{0, \alpha}}$$
with $$\Norm{u}{C^{0, \alpha}} = \underset{x \in \m}{\text{sup}} \norm{u}{\gz}  + \underset{x \in \m}{\text{sup}} \left( \underset{d_{\gz}(x,y) \infeg 1}{\text{sup}} \frac{\norm{\w{u}(x) - \w{u}(y)}{\w{g}}}{d_{\gz}(x,y)^{\alpha}} \right)$$
where $\w{u}$ et $\w{g}$ correspond to the tensors  $u$ and $g$ in an orthonormal basis.

\section{Analysis tools}

We first describe  a listing of  some more or less classic inequalities (see  eg. \cite{Andersson1993, CB2005,  HNT2009, BH2017b}).\\

\noindent
{\bf  H\"older inequalities :}
\begin{list}{$\bullet$}{}
\item
Let $p,q,r \in \nat$ be such that
$$ 1\infeg p\infeg q \infeg r \infeg \infty \;\; \text{and} \;\; \frac{1}{p} = \frac{1}{q} + \frac{1}{r} \, ,$$ then
\begin{equation}\label{h1}
\Norm{uv}{p } \infeg \Norm{u}{q } \; \Norm{v}{r }.
\end{equation}
\item Set  $\lambda \in \lbrack 0,1 \rbrack$ and  let $p,q,r \in \nat$ be such that
$$ 1\infeg p\infeg q \infeg r \infeg \infty \;\;  { {\text{and}}} \;\; \frac{1}{p} = \frac{\lambda}{q} + \frac{1- \lambda}{r} \, ,$$ then
\begin{equation}\label{h2}
\Norm{u}{p } \infeg \Norm{u}{q }^{\lambda} \; \Norm{u}{r }^{1- \lambda}.
\end{equation}
\end{list}
{\bf Sobolev embedding:}\\
For all $ 1\infeg p\infeg q< \infty $ , for all $k\supeg k'$, \; we have $\sob{k}{q}{} \subset \sob{k'}{p}{}$ ,\\
 {and} there exists a positive constant $c = c \, (\gz   ' , k, k',n, p,q)$ such that
$$\Norm{u}{k', p } \infeg c \, \Norm{u}{k, q }.$$
If $ q= \infty$, for all $ 1\infeg p\infeg \infty $; for all $ k \supeg k'$ , \; we get  $\sob{k}{\infty}{} \subset \sob{k'}{p}{}$ ,\\
 {and} there exists a positive constant $c$ such that
$$\Norm{u}{k', p } \infeg c \, \Norm{u}{k, \infty }.$$
{\bf H\"older embedding, Morrey's inequality:}\\
For all $ 0 < l+\alpha \infeg k+2-n/2$ , there exists a positive constant $c $ such that
\begin{equation}\label{incholdn}
\forall u \in \sob{k+2}{ 2}{}(\m) \; \; , \; \;\Norm{u}{C^{l, \alpha}_{}} \infeg c \, \Norm{u}{k+2, 2 }.
\end{equation}
{\bf Sobolev inequalities:}\\
Set $1\infeg p < \infty$ and let $k,j$ be integers.  { In each of the following cases, there exists a positive constant $c = c \,(\gz  , k,n,j,p,q)$ such that
for all $u \in \sob{j+k}{p}{}(\m)$,}
\begin{list}{$\bullet$}{}
\item If $pk<n, \; \Norm{u}{j, q } \infeg c \, \Norm{u}{j+k, p } \; , \; \; \forall \, p\infeg q \infeg \frac{np}{n-kp}$.\\
\item If $pk=n, \; \Norm{u}{j, q } \infeg c \, \Norm{u}{j+k, p } \; , \; \; \forall \, p\infeg q < \infty$.\\
\item If $pk>n, \; \Norm{u}{j, q } \infeg c \, \Norm{u}{j+k, p } \; , \; \; \forall \, p\infeg q \infeg \infty$.
\end{list}
\medskip
{\bf Ehrling inequality:}\\
Let  $j$ and $k$ be two integers such that $0 <j<k$. For all $\epsilon >0$, there exists a positive constant $C(\epsilon)$ such that
\begin{equation}\label{Ehrling}
\forall u \in \sob{k}{p}{} \; \; , \; \; \Norm{u}{j, p } \infeg \epsilon  \Norm{u}{k, p } + C(\epsilon)  \Norm{u}{p }.
\end{equation}
{\bf Rellich - Kondrachov Theorem:}\\
For all $k>k'$ , the inclusion $\sob{k}{2}{} \subset \sob{k'}{2}{}$ is compact.\\

{We will also need to control the norm of some products} (see eg. \cite{CB2005,  HNT2009, BH2017b} ).
\begin{lem}\label{lemmePiotr}
 Let $0 \leq m \leq l \leq k$, $q, p \in [1,+\infty]$, $kp > n$. Suppose that $(l, q)$ is
such that the Sobolev embedding $W^{k,p}\subset W^{l,q}$
holds. Then the product map
$$
W^{k-m,p}\times W^{l,q}\ni (f, g) \mapsto fg \in W^{l-m,q}
$$
is continuous. 
\end{lem}
\begin{remark} \em
For our applications we usually use that when  $k+2>n/2$,
\begin{equation}\label{19gk}
\Norm{uv}{k+2 ,2}+\Norm{u\mathring\nabla v}{k+1,2 } 
+\Norm{u\mathring\nabla\mathring\nabla v}{k,2 }+\Norm{\mathring\nabla u\mathring\nabla v}{k,2 }\infeg c \; \Norm{u}{k+2,2} \; \Norm{v}{k+2,2 } \; \; .
\end{equation}
\end{remark}
The following estimates  will also be useful.
\begin{lem}\label{Ehrlingadapte} Let $p\in [2,+\infty)$ such that  $
p< \frac{2n}{n-4}$ if $n>4$. Let $q\in[2,\frac{2n}{n-2})$.
For all $k\in\mathbb N$ and any $\epsilon>0$, there exists a positive constant $c = c \, (\epsilon)$ such that
\begin{eqnarray}
\forall u \in \sob{k+2}{2}{}(\m) \; \; , \; \; \Norm{u}{k,p } &\infeg& \epsilon \; \Norm{u}{k+2,2} +c (\epsilon)\Norm{u}{0,2 }. \label{uinfdk}\\
\forall u \in \sob{k+1}{2}{}(\m) \; \; , \; \; \Norm{u}{k,q } &\infeg& \epsilon \; \Norm{u}{k+1,2} + c (\epsilon)\Norm{u}{0,2 }. \label{u3dk}
\end{eqnarray}
\end{lem}
\begin{proof}
Let $r>p$ such that $r\leq \frac{2n}{n-4}$ if $n>4$. There exist $\lambda\in(0,1]$ such that
$$
\frac1p=\frac{\lambda}{2}+\frac{1-\lambda}r.
$$
By the Sobolev inequality (\ref{h2}) used for $u$ and its derivatives up to order $k$,
we find 
$$
\Norm{u}{k,p }\leq C_k \Norm{u}{k,2 }^\lambda \Norm{u}{k,r }^{1-\lambda},
$$
thus
$$
\Norm{u}{k,p }\leq C \Norm{u}{k,2 }^\lambda \Norm{u}{k+2,2 }^{1-\lambda},
$$
We can now use 
$$
a^\lambda b^{1-\lambda}\leq \lambda^\lambda (a+b)\leq (a+b),
$$
with $b=\epsilon  \Norm{u}{k+2,2 }$ and  $a=C^{1/\lambda}\epsilon^{-\frac{1-\lambda}{\lambda}}\Norm{u}{k,2 }$.
We obtain the desirate estimate by applying the Ehrling inequality  (\ref{Ehrling}) in order to control $a$.

Let $s>q$ such that $s\leq  \frac{2n}{n-2}$ then there exist $\lambda\in(0,1]$ such that
$$
\frac1q=\frac{\lambda}{2}+\frac{1-\lambda}s.
$$
As before, we have
$$
\Norm{u}{k,q }\leq C_k \Norm{u}{k,2 }^\lambda \Norm{u}{k,s }^{1-\lambda},
$$
so 
$$
\Norm{u}{k,q }\leq C \Norm{u}{k,2 }^\lambda \Norm{u}{k+1,2 }^{1-\lambda},
$$
and we conclude in the same way
\end{proof}
\begin{remark}\label{produitadapte} \em
For our applications, me will often use the lemma \ref{Ehrlingadapte} combined with the inequality
$$
 \Norm{uv}{k,2 }\leq  \Norm{u}{k,p } \Norm{v}{k,2 },
$$
true when $kp>n$ (see lemma \ref{lemmePiotr} ) so we need $\frac n k<\frac{2n}{n-4}$ if $n>4$ for $p$ to exist, it  is possible precisely  when $k+2>\frac n2$.
We will also frequently use the lemma \ref{Ehrlingadapte} together with
$$
 \Norm{uv}{k,2 }\leq  \Norm{u}{k,q} \Norm{v}{k+1,2 },
$$
true when $(k+1)q>n$ (see lemma \ref{lemmePiotr} ) so we need $\frac n {k+1}<\frac{2n}{n-2}$ for $q$ to exist,  which  is true again exactly when $k+2>\frac n2$.
\end{remark}

We now bring in some operators and their properties.\\

\noindent {\bf Elliptic operators.}$\mbox{ }$\\
Here we recall a classical results about elliptic operators (see eg.  \cite{Andersson1993} ).
Let $B_{1}$ and $B_{2}$ be two tensor bundles over a  compact manifold $(\m, {\gz})$ and $A: \cinf(B_{1}) \rightarrow \cinf(B_{2})$ be a partial differential linear operator of order $m$ define{d} by
\begin{equation}\label{defAop}
A= \sum_{|\alpha| \infeg m} a_{\alpha}\nz^{\alpha}.
\end{equation}

We say {that} $A$ is an elliptic operator if
\begin{itemize}
\item[\textbullet] For all $\alpha$ such that $|\alpha|=m$ , for all
$\xi^{\alpha} = \xi_{1}^{\alpha_{1}} \ldots \xi_{n}^{\alpha_{n}} \ne 0$,\\
 $a_{\alpha} \xi^{\alpha} : B_{1} \rightarrow B_{2}$ is a tensor bundles isomorphism.
\item[\textbullet] There exists two constants $c_{1}$ and $c_{2}$ such that,  for all $\xi$, 
$$\Norm{a_{\alpha} \xi^{\alpha}}{ {\gz}} < c_{1} \; \norm{\xi^{\alpha}}{ {\gz}} \hspace{.8cm} \text{and} \hspace{.8cm} \Norm{(a_{\alpha} \xi^{\alpha})^{-1}}{ {\gz}} < c_{2} \; \norm{\xi^{\alpha}}{ {\gz}} .$$
\end{itemize}
Let us recall the following  classical elliptic estimate.
\begin{lem}\label{lemestellip}
For every elliptic operator $A$ of order $m$, there exists  a positive constant $c = c \, (\gz , k)$ such that for all $u\in \leb{1}{}$ such that $Au\in \sob{k} {2}{}$ then $u\in \sob{k+m} {2}{}$ with :
\begin{equation}\label{sbelemme}
\Norm{u}{k+m,2} \infeg c \, \left(\Norm{Au}{k,2} + \Norm{u}{1}\right).
\end{equation}
 Moreover $A: \sob{k+m} {2}{}\rightarrow \sob{k} {2}{}$ is semi-Fredholm, $\ie A$ has finite dimensional kernel and closed range.
\end{lem}
\noindent
{\bf {The} Killing operator }

We study  the Killing operator $\Sz$ defined on 1-forms by
\begin{equation}
\Sz(Y)_{ij} = \demi (\nz_{i}Y_{j} + \nz_{j}Y_{i}) = \nz_{(i}Y_{j)}.
\end{equation}
{This operator plays an important role when studying 
 the formal adjoint to the constraint  operator (also called the KID's operator).}
The goal  of this section is to recall a   Korn inequality.
\begin{lem}\label{SmajoreX} Let $k\geq -1$. There exist a constants $C>0$, such that for any one form $X\in L^2$ such that $\Sz(X)\in H^{k+1}$ then $X\in H^{k+2}$ and   
$$
\Norm{X}{k+2,2} \infeg C \,( \Norm{\Sz(X)}{k+1,2}+\Norm{X}{0,2}).
$$
\end{lem}
\begin{proof}
For $k=-1$, we compute 
$$
2\int |\Sz(X)|^2=
2\int \Sz(X)_{ij}\nz^iX^j=\int|\nz X|^2+\int \nz_jX_i\nz^iX^j=\int|\nz X|^2-\int X_i\nz_j\nz^iX^j,
$$
we deduce 
$$
2\int |\Sz(X)|^2=\int|\nz X|^2+\int div(X)^2+Ric(\gz)(X,X),
$$
so we  easily obtain
\begin{equation}\label{kvmu}
\Norm{X}{1,2} \infeg C \,( \Norm{\Sz(X)}{0,2}+\Norm{X}{0,2}).
\end{equation}
If $k\in\mathbb N$ we also use 
 the following identity  (see e.g. equation (29) of \cite{Bartnik2005} for example)
\begin{equation}\label{b29}
\nz^{2}_{kj}X_{i} :=\nz_k\nz_jX_i= Riem \, \gz_{ijkl} X^{l} + \nz_{k} \Sz (X)_{ij} + \nz_{j} \Sz (X)_{ik} - \nz_{i} \Sz (X)_{jk}.
\end{equation}
This leads to
\begin{eqnarray}\label{b29bis}
\Norm{\nz^{2}X}{k,2} &\infeg& \Norm{Riem \,\gz \; X}{k,2} + c \, \Norm{\nz \Sz (X)}{k,2} \nonumber\\
 &\infeg& c \, \Norm{X}{k,2} + c \, \Norm{\nz \Sz (X)}{k,2}.
\end{eqnarray}
The equations (\ref{kvmu}) , (\ref{b29bis})  and the Ehrling inequality (\ref{Ehrling}) close the proof.
\end{proof}
\begin{remark} \em
During the proof the following interesting operator apear 
$$
\Uz(X)_{kji} :=\nz_k\nz_jX_i- Riem \, \gz_{ijkl} X^{l} = \nz_{k} \Sz (X)_{ij} + \nz_{j} \Sz (X)_{ik} - \nz_{i} \Sz (X)_{jk}.
$$
\end{remark}
\noindent
{\bf A  shifted Hessian operator}

When studying the KID's, a natural operator acting on function apear :
$$
 L(N)_{ij}:=\n_{i} \n_{j} N - g_{ij}\laplag N - \lbrack R_{ij} - \tfrac{1}{2}  (R(g) - 2 \Lambda) g_{ij} \rbrack N.
$$
We thus define the operator 
$$
T(N)=\nabla \nabla N -\lbrack Ric(g)-\frac{1}{2(n-1)}(R(g)+2\Lambda)g\rbrack N,
$$
so $L(N)=T(N)-\tr_g(T(N))g$.
Using the Ehrling inequality (\ref{Ehrling}) we immediately obtain :
\begin{lem}\label{TmajoreN}
Let $k\in\mathbb N$. There exist a constant $C>0$  such that  If $N\in L^2$ is such that $\Tz(N)\in H^k$ then $N\in H^{k+2}$ and 
$$
\Norm{N}{k+2,2}{}\leq C(\Norm{\Tz(N)}{k,2}{} +\Norm{N}{0,2}{}).
$$ 
\end{lem}

\section{The constraint operator }\label{sectionopecontrainte}

Let $\m$ be a $n$-dimensional connected smooth compact oriented manifold. We fix a smooth Riemannian metric $\gz$ on $\m$. Let $\tau\in\R$ and let
$$
\Kz=\tau \gz.
$$
 We consider $\m$ as a  {spacelike  hypersurface} of a $(n+1)$-dimensional Lorentzian manifold $(\N, \gamma)$, from now on called spacetime. We will identify  the two manifolds by different indices: Latin indices will take values from $1$ to $n$ and are spatial indices whereas Greek indices will take values from $0$ to $n$ and are spacetime indices. $K$ is the second fundamental form of $\m$ in $\N$ defined by
\begin{equation}\label{defk}
K(X,Y) = \gamma (X, \precig{\gamma}\n_{Y} \vec n) ,
\end{equation}
where $\precig{\gamma}\n$ is the spacetime connection on $T\N$ , $X,Y \in T\m$ and $\vec n$ is the future-directed unit normal to $\m$ in $\N$. It is convenient to consider the conjugate momentum $\pi$ as a reparametrisation of $K$
\begin{equation}\label{defpi}
\pi^{ij} = \w{\pi}^{ij} \sqrt{g} \; \; \; \text{with} \; \; \w{\pi}^{ij} = K^{ij} -  {\tr}_{g}K g^{ij},
\end{equation}
where $\sqrt{g}$ is the  (relative) volume measure of the metric $g$ :
\begin{equation*}\label{defsqrtg}
\sqrt{g} = \frac{\sqrt{det(g)}}{\sqrt{det(\gz)}} \, ,
\end{equation*}
identified with the volume form 
\begin{equation*}
d\mu(g)= \sqrt{g} \, d\mu(\gz).
\end{equation*}
The field $\w{\pi}$ is a section of the bundle $S^{2}T\m$ of symmetric bilinear forms on $\m$, whereas $\pi$ is a section of the bundle $\w{S}=S^{2}T\m \otimes \Lambda^{n}T^{*}\m$ of symmetric $2$-tensors-valued densities ($n$-forms) on $\m$.

We define the constraint operator $\phibf=(\phio , \phii)$ as follows:
\begin{eqnarray}
\phiogpi &:=& \left(R(g) - 2 \Lambda - \norm{K}{g}^{2} + ( {\tr}_{g}K)^{2}\right) \sqrt{g} \nonumber\\
&=& \left(R(g) - 2 \Lambda\right) \sqrt{g} - \left(\norm{\pi}{g}^{2} - \tfrac{1}{n-1} ( {\tr}_{g}\pi)^{2}\right) \slash \sqrt{g}.\\
\phiigpi &:=& 2 (\n^{j}K_{ij} - \n_{i}( {\tr}_{g}K) )\sqrt{g} \nonumber\\
&=& 2 g_{ij} \n_{k} \pi^{jk} =  2 g_{ij} \n_{k} \w{\pi}^{jk} \sqrt{g}. \label{ECCG}
\end{eqnarray}

If the spacetime satisfies Einstein's equations, the normalisation chosen insures that the constraint operator and the energy-momentum tensor are related by $$ \phibf_{\alpha} = 16 \pi G T_{\vec n\alpha} \sqrt{g} \, ,$$
where $G$ is Newton's gravitational constant. $\xi = (N,X^{i})$ is the \emph{lapse-shift} associated to the spacetime foliation. 

We denote by $\T := T\N_{|\m}$ the spacetime tangent bundle restricted  to $\m$. The following spaces will be used along the paper (recall we would like an easy adaptation to non compact setting)
\begin{equation*}
\begin{array}{ccl}
\G& := &\sob{k+2}{2}{} (\S).\\
\K& := &\lbrace \pi : \pi -\piz \in \sob{k+1}{2}{} (\w{\S}) \rbrace.\\
\G^{+}& := &\lbrace g : g-\gz \in \G , g>0\rbrace.\\
\G^{+}_{\lambda}& := &\lbrace g \in \G^{+} : \lambda \gz < g < \lambda^{-1} \gz\rbrace \; , \; 0 < \lambda < 1.\\
\L^{*}& := &\sob{k}{2}{}(\T^{*} \otimes \Lambda^{n} T^{*}\m) \, \, \text{is the dual space of} \, \, \L := \sob{-k}{2}{}(\T).
\end{array}
\end{equation*}
From $(\ref{incholdn})$, tensors in $\G$ are H\"older-continuous
and thus, matrices inequalities in  $\G^{+}$  are satisfied pointwise. In particular, for all metric $g \in \G^{+}_\lambda $, metrics $g$ and $\gz$ are equivalent in the following sense :
\begin{equation}\label{holdcontcg}
 \forall x \in \m \, , \forall v \in T_x\m\;,\;\;
\lambda \gz_{ij}(x) \, v^{i}v^{j} < g_{ij}(x) \, v^{i}v^{j} < \lambda^{-1} \gz_{ij}(x) \, v^{i}v^{j} \; .
\end{equation}
$\F = \Gplus \times \K$ will be the phase space of the constraint operator $\phibf$. We will use $(g,\pi)$ as well as $(g,K)$ to denote coordinates on $\F$.\\[.5cm]
Let $\cdez{}{}{}$ and $\nz$ (resp. $\cde{}{}{}$ and $\n$) be respectively  the Christoffel symbols and the Levi-Civita connection of $\gz$ (resp. $g$). We define
\begin{equation}\label{defA}
\ade{i}{k}{j} = \cde{i}{k}{j}-\cdez{i}{k}{j}.
\end{equation}
It is well known that
\begin{equation}\label{A}
\ade{i}{k}{j}  = \demi g^{kl} (\nz_{i} g_{jl} + \nz_{j} g_{il} - \nz_{l} g_{ij}).
\end{equation}
The scalar curvature of $g$ can be formulate using $\nz$ and $\ade{i}{j}{k}$ (see eq. (21) of \cite{Bartnik2005}):

\begin{eqnarray}\label{courbscalcov}
R(g) &=& g^{jk} Ric\, \gz_{jk} + g^{jk}(\nz_{i} \ade{j}{i}{k} - \nz_{j} \ade{i}{i}{k} + \ade{j}{l}{k}\ade{i}{i}{l} - \ade{j}{i}{l}\ade{k}{l}{i}) \nonumber\\
&=& g^{jk} Ric\, \gz_{jk} + Q(g^{-1}, \nz g) + g^{ik} g^{jl}(\nz^{2}_{ij} g_{kl} - \nz^{2}_{ik} g_{jl})
\end{eqnarray}
where Q is a sum of quadratic terms in $g^{-1}, \nz g$.

This result relies on the following fact:

\begin{equation}\label{riccicov}
Ric \, g_{jk} - Ric \, \gz_{jk} = \nz_{i} \ade{j}{i}{k} - \nz_{j} \ade{i}{i}{k} + \ade{j}{\mu}{k}\ade{i}{i}{\mu} - \ade{j}{i}{\mu}\ade{k}{\mu}{i}
\end{equation}

Here we show $\phibf$ is a well-defined mapping between the Hilbert spaces $\F$ and $\L^{*}$.
\begin{prop}{}\label{phidef}  Let $ 0 < \lambda < 1$. There exists a positive constant $c = c \, (\lambda, \gz)$ such that
for all $(g,\pi) \in \G^{+}_{\lambda}\times \K$ , 
\begin{eqnarray}\label{16g}
\Norm{\phiogpi}{k,2} &\infeg& c \, \big(1 + \Norm{g-\gz}{k+2,2}^{2} + \Norm{\pi- \piz}{k+1,2}^{2}\big)\\
\Norm{\phiigpi}{k,2} &\infeg& c \, \Big(\Norm{\nz (\pi-\piz)}{k,2} +\Norm{\nz g}{k+1,2} (1+ \Norm{\pi- \piz}{k+1,2})\Big)
\end{eqnarray}
\end{prop}
\textbf{Proof}: From $R(g)$ expression $(\ref{courbscalcov})$, we get
\begin{eqnarray*}
\phiogpi &=& (R(g) - 2 \Lambda) \sqrt{g} - (\norm{\pi}{g}^{2} - \tfrac{1}{n-1} ( {\tr}_{g}\pi)^{2}\rbrack \slash \sqrt{g}\\
&=& \lbrack R(g)- R(\gz) + R(\gz) - 2 \Lambda + n(n-1) \tau^{2}\rbrack \sqrt{g} - \lbrack \norm{\pi- \piz}{g}^{2} + 2(\pi - \piz)_{ij}\piz^{ij} \rbrack \slash \sqrt{g} \\
&& + \tfrac{1}{n-1} \lbrack ( {\tr}_{g}(\pi- \piz))^{2} + ((g- \gz)_{ij}\piz^{ij})^{2} + 2 (g- \gz)_{ij}\piz^{ij} {\tr}_{\gz}\piz + 2 \,  {\tr}_{g}(\pi- \piz) \,  {\tr}_{g}\piz \rbrack \slash \sqrt{g}.
\end{eqnarray*}
Since $g \in \G^{+}_{\lambda}$, we can use (\ref{holdcontcg}) and from Cauchy-Schwarz inequality and $2ab \infeg a^{2} + b^{2}$
\begin{eqnarray*}
\norm{\phiogpi}{\gz} &\infeg& \lbrack \norm{R(g) - R(\gz)}{\gz} + \norm{R(\gz) - 2 \Lambda + n(n-1) \tau^{2}}{\gz} \rbrack \sqrt{g}\\
&& + c \, \lbrack 1 + \norm{\pi- \piz}{g}^{2}  + \norm{g- \gz}{g}^{2} \rbrack \slash \sqrt{g}.
\end{eqnarray*}
From (\ref{riccicov}), $Ric \, g - Ric \, \gz \simeq \nz A + A^{2} \simeq (\nz g) ^{2} + g \nz^{2}g + g^{-2} (\nz g) ^{2}$.\\
Using lemma \ref{lemmePiotr} and remark \ref{produitadapte},
\begin{eqnarray*}
\Norm{Ric \, g - Ric \, \gz}{k,2} 
&\infeg& c \, \Norm{g - \gz}{k+2,2},
\end{eqnarray*}
 and 
\begin{eqnarray*}
\Norm{R(g) - R(\gz)}{k,2} 
&\infeg& c \, \Norm{g - \gz}{k+2,2}.
\end{eqnarray*}
In particular, we have 
\begin{eqnarray}
Ric \, g - Ric \, \gz \in W^{k,2} , \label{integdifric}\\
R(g) - R(\gz) \in W^{k,2} . \label{integdifscal}
\end{eqnarray}
Thanks to (\ref{integdifscal}) and lemma \ref{lemmePiotr}, we obtain the estimate
\begin{eqnarray*}
\Norm{\phiogpi}{k,2}&\infeg& c \left( 1 + \Norm{(\pi- \piz)^{2}}{k,2} +  \Norm{(g- \gz)^{2}}{k,2} \right)\\
&\infeg& c \left( 1 + \Norm{\pi- \piz}{k+1,2}^{2} +  \Norm{g- \gz}{k+2,2}^{2} \right),
\end{eqnarray*}
hence $\phiogpi \in \L^{*}$.\\
For $\phiigpi$ , using (\ref{defA}), we have  
\begin{equation*}
\phiigpi = 2 g_{ij}(\nz_{k} (\pi- \piz)^{jk} + \ade{k}{j}{l} \, (\pi-\piz)^{kl}) + \ade{k}{j}{l} \, \piz^{kl}).
\end{equation*}
Considering (\ref{A}), $\phiigpi$ is of the form
\begin{equation}\label{phiidef}
\phiigpi \simeq g (\nz (\pi- \piz) + g^{-1} \nz g \, (\pi- \piz) +  g^{-1} \nz g \, \piz),
\end{equation}
thus again by lemma \ref{lemmePiotr}  and remark \ref{produitadapte},
\begin{eqnarray*}
\Norm{\phiigpi}{k,2} &\infeg& c \, (\Norm{\nz (\pi- \piz)}{k,2} + \Norm{\nz g \, (\pi- \piz)}{k,2} + \Norm{\nz g \, \piz}{k,2})\\
 &\infeg& c \, (\Norm{\nz (\pi- \piz)}{k,2} + \Norm{\nz g}{k+1,2} \Norm{\pi- \piz}{k+1,2} + \Norm{\nz g}{k+1,2}\Norm{\piz}{k+1,2}) \\
 &\infeg& c \, \big(\Norm{\nz (\pi- \piz)}{k,2} + \Norm{\nz g}{k+1,2} (1+ \Norm{\pi- \piz}{k+1,2})\big). \cqfd
\end{eqnarray*}
We now bring in the principal result of the section.
\begin{prop}{}
The map $\phibf : \F \rightarrow \L^{*}$ is a smooth map between Hilbert spaces.
\end{prop}
\textbf{Proof}: We recall the proof of \cite{Bartnik2005} for completeness. From Proposition \ref{phidef}, 
$$\Norm{\phigpi}{\L^{*}} \infeg c (1 + \Norm{g-\gz}{\G}^{2} + \Norm{\pi-\piz}{\K}^{2}),$$ $\ie \phibf$ is locally bounded on $\F$. The polynomial structure of the constraint operator allows us to show $\phibf$ is smooth, $\ie$ indefinitely differentiable in a Fr\'echet sense. From the expression (\ref{courbscalcov}) of scalar curvature and given (\ref{phiidef}) , $\phibf$ can be expressed as
$$\phigpi = F (g, g^{-1}, \sqrt{g} , 1 \slash \sqrt{g} , \nz g, \nz^{2} g, \pi , \nz \pi) \, , $$
where $F= F(a_{1}, \ldots ,a_{8})$ is a polynomial function quadratic in $a_{5}$ and $a_{7}$ and linear in the remaining parameters. The map $g \mapsto (g, g^{-1}, \sqrt{g} , 1 \slash \sqrt{g})$ is analytic on the space of positive definite matrices and the maps $g \mapsto \nz g \, , \, g \mapsto \nz^{2} g$ and $\pi \mapsto \nz \pi$ are bounded linear, thus smooth, from $\F$ to $\L^{*}$, which are Hilbert spaces. A result from Hille \cite{Hille1957}  on locally bounded polynomial functional shows $\phibf$ admit continuous Fr\'echet-derivatives of all orders.$\cqfd$\\

\section{Linearised constraint and KID's operator}

The set $\contrainte = \lbrace (g, \pi) \in \G^{+} \times \K : \phigpi = 0 \rbrace := \phibf^{-1}(\lbrace 0 \rbrace) \subset \F$ is the set of initial data for the vacuum Einstein's equations. To prove that $\contrainte$ is a submanifold of $\F$, we show that $0$ is a regular value of $\phibf$ , so we  {are interested} in the surjectivity of the differential of $\phibf$, also related to the injectivity of its adjoint.
We recall the expression of the linearization of {the constraint operator} $\phibf$ and its formal adjoint, the KID's operator (see \cite{Bartnik2005} or \cite{Fischer1979} for example).
\begin{prop}\label{corestiadjL2}
{The differential of the constraint operator $\phibf$ at $(g,\pi)$ in the direction $(h,p)$ is}
\begin{eqnarray}
D \phiogpi .(h,p) &=& (\n^{i} \n^{j} h_{ij} - \laplag  {\tr}_{g}h) \sqrt{g} - h_{ij} \lbrack R^{ij} - \tfrac{1}{2}  (R(g) - 2 \Lambda) g^{ij} \rbrack \sqrt{g}\nonumber\\
&& + h_{ij} \big(\tfrac{2}{n-1}  {\tr}_{g}\pi \pi^{ij} - 2 \pi^{i}_{\, k} \pi^{kj} + \tfrac{1}{2} \norm{\pi}{g}^{2} \, g^{ij} - \tfrac{1}{2(n-1)} ( {\tr}_{g}\pi)^{2} g^{ij} \big) \slash \sqrt{g} \nonumber\\
&& + p^{ij} (\tfrac{2}{n-1} {\tr}_{g}\pi g_{ij} - 2 \pi_{ij}) \slash \sqrt{g}. \label {def0}\\[.3cm]
D \phiigpi .(h,p) &=& \pi^{jk}(2 \n_{k}h_{ij}- \n_{i} h_{jk}) + 2 h_{ij} \n_{k}\pi^{jk} + 2 g_{ik} \n_{j} p^{jk}. \label{defi}
\end{eqnarray}
\end{prop}

Using notations of \cite{Bartnik2005} ,  we define
\begin{eqnarray*}
\delta_{g}\delta_{g} h &=& \n^{i}\n^{j} h_{ij},\\
E^{ij} &=& R^{ij} - \tfrac{1}{2}  (R(g) - 2 \Lambda) g^{ij},\\
\Pi^{ij} &=& \big(\tfrac{2}{n-1}  {\tr}_{g}\pi \pi^{ij} - 2 \pi^{i}_{\, k} \pi^{kj} + \tfrac{1}{2} \norm{\pi}{g}^{2} \, g^{ij} - \tfrac{1}{2(n-1)} ( {\tr}_{g}\pi)^{2} g^{ij} \big) \slash (\sqrt{g})^{2}.
\end{eqnarray*}
We can express $D \phibf$ in the  following form
\begin{equation}
D\phigpi.(h,p) =
\left[
\begin{array}{cc}
\sqrt{g}(\delta_{g}\delta_{g} - \laplag  {\tr}_{g} + \Pi - E) & -2 K\\
\hat{\pi}\n + 2 \delta_{g} \pi & 2 \delta_{g}
\end{array}
\right]
\left[
\begin{array}{c}
 h\\
p
\end{array}
\right] \, ,
\end{equation}
with $\; \hat{\pi} \n h = \hat{\pi}_{i}^{jkl} \n_{j} h_{kl} = (\pi^{jk} \delta_{i}^{l} + \pi^{jl} \delta_{i}^{k} - \pi^{kl} \delta_{i}^{j})\n_{j} h_{kl}$.\\[.5cm]
To prove surjectivity of the differential of $\phibf$ , we study the injectivity of the adjoint operator. Integrating by parts and ignoring boundary terms leads (cf. \cite{Fischer1979} for example) to the expression of the formal $L^{2}(d \mu (\gz))$-adjoint of $D \phigpi$, also called the KID's operator:
\begin{equation*}
\int_{\m} D \phigpi .(h,p) \, (N,X) = \int_{\m} (h,p) \bullet D \phigpis (N,X).
\end{equation*}
The detail of the product is given by the following equalities
\begin{eqnarray*}\label{adjphiocg}
(h,p) \bullet D \phiogpis N &=& h_{ij} \lbrack \n^{i} \n^{j} N - g^{ij}\laplag N - \lbrack R^{ij} - \tfrac{1}{2}  (R(g) - 2 \Lambda) g^{ij} \rbrack N \rbrack \sqrt{g}\\
&&+ N h_{ij} \big(\tfrac{2}{n-1}  {\tr}_{g}\pi \pi^{ij} - 2 \pi^{i}_{\, k} \pi^{kj} + \tfrac{1}{2} \norm{\pi}{g}^{2} \, g^{ij} - \tfrac{1}{2(n-1)} ( {\tr}_{g}\pi)^{2} g^{ij} \big) \slash \sqrt{g}\\
&& +N p^{ij} (\tfrac{2}{n-1}  {\tr}_{g}\pi g_{ij} - 2 \pi_{ij}) \slash \sqrt{g}.\\
(h,p) \bullet D \phiigpis X^{i} &=& h_{ij}(X^{k} \n_{k} \pi^{ij} +   \n_{k} X^{k} \pi^{ij} - 2  \n_{k}  X^{(i} \pi^{j)k}) -2 p^{ij} \n_{(i} X_{j)}.
\end{eqnarray*}

Then we can put the KID's operator  $D \phibf^*$ in the matrix  form
\begin{equation}\label{dphis}
D\phigpis.(N,X) =
\left[
\begin{array}{cc}
\sqrt{g}(\n^{2} - g \laplag + \Pi - E) & \n \pi - \hat{\pi}\n\\
 -2 K & - \L_{g}
\end{array}
\right]
\left[
\begin{array}{c}
 N\\
X
\end{array}
\right] \, ,
\end{equation}
with
\begin{eqnarray*}
(\n \pi - \hat{\pi}\n)X &=& \L_{X}\pi = \n_{X}\pi^{ij} - \hat{\pi}_{l}^{kij} \n_{k} X^{l},\\
\L_{g}(X)&=& \L_{X}g = 2 \, \n_{(i}X_{j)} = 2 \, S(X).
\end{eqnarray*}
$D\phigpis_{1}.\xi$ and $D\phigpis_{2}.\xi \;$  {will denote} the two components of $D\phigpis$ in (\ref{dphis}).\\
 We will denote by $\sob{k}{2}{}  \xi$ any terms of the form $u \, \xi$ such that $\Norm{u}{k,2} \infeg C$, where $C$ is a constant depending on $\gz$ and $\Norm{(g,\pi)}{\F}$.
We have
\begin{eqnarray}
D\phigpis_{1}.\xi &=& \lbrack \n_{i} \n_{j} N - g_{ij}\laplag N + (\Pi_{ij} - E_{ij}) \rbrack N \rbrack \sqrt{g} + (\n \pi - \hat{\pi}\n)X \nonumber\\
&=& D \phibf (g,0)^{*} \, (N,0) + \Pi_{ij} N \sqrt{g} + (\n \pi - \hat{\pi}\n)X \label{expdphi1dphi0} ,\\[.3cm]
(\n \pi - \hat{\pi}\n)X&=&X^{k}\n_{k}\pi_{ij} - (\pi^{k}_{\, i} \delta_{lj} + \pi^{k}_{\, j} \delta_{li} - \pi_{ij} \delta_{l}^{k}) \n_{k} X^{l} \nonumber\\
&=& \sob{k}{2}{} X +  \sob{k+1}{2}{} \nz X +(n-1) \tau (2 \Sz (X) - \gz \, {\tr}_{\gz}\Sz(X)), \nonumber\\[.3cm]
\Pi(g, \pi) N &=&\sob{k}{2}{} N + \Pi(\gz, \piz) N \nonumber\\
&=& \sob{k}{2}{} N  - \tdemi(n-1)(n-4)\tau^{2} \gz N.\nonumber
\end{eqnarray}
If $(g,\pi)\in \F$ we have
\begin{equation}\label{integPi}
\Pi(g, \pi) -\Pi(\gz, \piz) =\Pi(g, \pi) + \tdemi(n-1)(n-4)\tau^{2} \gz \in \sob{k}{2}{},
\end{equation}
and 
\begin{equation}\label{integE}
E -k (n-1) \gz -\tdemi n(n-1)\tau^{2} \gz \in \sob{k}{2}{},
\end{equation}
(this last quantity being equal to $E-\mathring E$ if $Ric(\gz)=k(n-1)\gz$).

On one hand, we find
\begin{eqnarray}
D\phigpis_{1}.\xi \slash \sqrt{g} &=& \n^{2} N - g\laplag N -k (n-1) \gz {N} +\lbrack \Pi + \tdemi(n-1)(n-4)\tau^{2} \gz \rbrack N \nonumber\\
&& +(n-1) \tau (2 \Sz (X) - \gz \, {\tr}_{\gz}\Sz(X))  - \lbrack E-k (n-1) \gz -\tdemi n(n-1)\tau^{2} \gz \rbrack N \rbrack \nonumber\\
&& - (n-1)(n-2)\tau^{2} \gz N + \sob{k}{2}{} \xi +  \sob{k+1}{2}{} \nz X \nonumber\\
&=& \n^{2} N - g\laplag N -k(n-1) {N}+ (n-1) \tau (2 \Sz (X) - \gz\,  {\tr}_{\gz}\Sz(X)) \nonumber\\
&& - (n-1)(n-2)\tau^{2} \gz N + \sob{k}{2}{} \xi +  \sob{k+1}{2}{} \nz X. \label{expdphi1}
\end{eqnarray}
On the other hand, we can write
\begin{eqnarray*}
D\phigpis_{2}.\xi &=& -2 KN - 2 S(X)\\
&=& - 2 (\Sz (X) + \tau \gz N) + \sob{k+1}{2}{} \xi.
\end{eqnarray*}
From the definition of the operator $T = \n^{2}N -Ng$ and the expression of $\Sz$ related to $D\phigpis_{2}.\xi$, we obtain
\begin{eqnarray}\label{dphis12}
D\phigpis_{1}.\xi \slash \sqrt{g}&=& T- g\, {\tr}_{g}T + (n-1) \tau (2 \Sz (X) - \gz\,  {\tr}_{\gz}\Sz(X))\nonumber\\
&& - (n-1)(n-2)\tau^{2} \gz N + \sob{k}{2}{} \xi +  \sob{k+1}{2}{} \nz X. \label{dphis1}\\
D\phigpis_{2}.\xi &=& - 2 (\Sz (X) + \tau \gz N) + \sob{k+1}{2}{} \xi. \label{dphis2}
\end{eqnarray}

It is useful to restructure $D\phibfs$ into the operator $\Ps$ defined by
\begin{eqnarray}\label{Ps}
\Ps(\xi) = \Pgpis(\xi) &=&
\left[
\begin{array}{c}
g^{1 \slash 4}\left(\n^{i}\n_{j}N - \delta^{i}_{\, j} \laplag N + (\Pi^{i}_{\, j} - E^{i}_{\, j})N\right) + g^{-1\slash 4}\L_{X}\pi^{i}_{\, j}\\
 - 2 g^{-1 \slash 4} \n_{l}( K^{i}_{\, j} N +  S(X)^{i}_{\, j})
\end{array}
\right]
 \nonumber \\
&=& \zeta \circ
\left[
\begin{array}{cc}
1 & 0\\
0 & \n
\end{array}
\right]
\circ D\phigpis \xi \, ,
\end{eqnarray}
where $g^{1 \slash 4} = (det(g) \slash det(\gz))^{1 \slash 4} \, d\mu(\gz)$ is a density of weight $\tdemi$ and
\begin{equation}\label{defzeta}
\zeta = \zeta(g) =
\left[
\begin{array}{cc}
g^{-1 \slash 4} g_{jk} & 0\\
0 & g^{1 \slash 4} g^{ik}
\end{array}
\right].
\end{equation}
Finally, we can put $\Pgpis(\xi)$ into the form
\begin{equation}\label{defPs}
\Pgpis(\xi) =
\left(
\begin{array}{c}
g^{-1 \slash 4} \, D\phigpis_{1}.\xi  \\
 g^{1 \slash 4} \, \n D\phigpis_{2}.\xi 
\end{array}
\right).
\end{equation}
Expression (\ref{Ps}) of $\Ps$ allows us to rewrite the $L^{2}(d \mu (\gz))$-adjoint of $\Ps$ as follows
\begin{equation}\label{defP}
\Pgpi = D\phigpi \circ 
\left[
\begin{array}{cc}
1 & 0\\
0 & - \delta_{g}
\end{array}
\right]
\circ \zeta \, ,
\end{equation}
with $\delta_{g}q = \n^{l}(q_{l}^{ij})$ so that $P(f^{i}_{\, j},q_{li}^{\;j}) = D\phibf(f_{ij},q_{l}^{ij})$ and so the composition $P \Ps$ is well defined.\\

\section{Elliptic estimates for the KID's operator}
In this section, we gather elliptic estimates satisfied by the adjoint operator $D \phibfs$.
We start with :
\begin{prop}{}\label{prop3.3g}
 If $k+2>n/2$, there exists a positive constant  
$C = C \, (\gz, \lambda , \Norm{g}{\F})$ such that the following elliptic estimate hold:
 $\forall \xi \in \sob{k+2}{2}{} (\T)$ ,
\begin{equation}\label{35cg}
\Norm{\xi}{k+2,2} \infeg c \, \big(\Norm{D \phigpis_{1}.\xi }{k,2} + \Norm{D \phigpis_{2}.\xi }{k+1,2} \big) + C \, \Norm{\xi}{0,2} \, ,
\end{equation}
\end{prop}
\textbf{Proof}: Considering expression (\ref{dphis2}) of $\Sz$ as a function of $D\phigpis_{2}.\xi$,
\begin{eqnarray}\label{T-gtrT}
T- g\, {\tr}_{g}T &=& D\phigpis_{1}.\xi \slash \sqrt{g} + (n-1)\tau \big(D\phigpis_{2}.\xi - \tdemi \gz \, {\tr}_{\gz} D\phigpis_{2}.\xi \big) \nonumber\\
&& + \sob{k}{2}{} \xi +  \sob{k+1}{2}{} \nz X.
\end{eqnarray}
From lemma \ref{TmajoreN}, equation (\ref{h1}), and the  remark  \ref{produitadapte} we have
\begin{eqnarray*}
\Norm{N}{k+2,2} &\infeg& c \, \big( \Norm{D\phigpis_{1}.\xi}{k,2} +  \, (n-1) |\tau| \, (1+ \tfrac{n^{2}}{4})^{\demi} \Norm{D\phigpis_{2}.\xi}{k,2} \big)\\
&& + C \, (\Norm{\xi}{k,p} + \Norm{\nz \xi}{k,q}+\Norm{N}{0,2}).
\end{eqnarray*}
Using (\ref{uinfdk}), (\ref{u3dk}) and Sobolev embedding, there exists a positive constant\\  
$C = C \, (\gz, \lambda , \Norm{g}{\F})$ such that
\begin{eqnarray}\label{estimeeN2}
\Norm{N}{k+2,2} &\infeg& c \, \big( \Norm{D\phigpis_{1}.\xi}{k,2} +  \, (n-1) |\tau| \, (1+ \tfrac{n^{2}}{4})^{\demi} \Norm{D\phigpis_{2}.\xi}{k,2} \big) \nonumber\\
&& + \epsilon \; \Norm{\xi}{k+2,2} + C \, \Norm{\xi}{0,2}.
\end{eqnarray}

Now from (\ref{dphis2}) and (\ref{uinfdk}), we get the estimate
\begin{equation}\label{estsx}
\Norm{\Sz (X)}{k+1,2} \infeg \tfrac{1}{4}\Norm{D\phigpis_{2}.\xi}{k+1,2} + n |\tau| \, \Norm{N}{k+1,2} + \epsilon \Norm{\xi}{k+2,2} + C \, \Norm{\xi}{0,2}.
\end{equation}
Consequently, using the lemma \ref{SmajoreX} there exists a constant $C$ depending on $\gz, \lambda, \epsilon$ and $\Norm{(g,\pi)}{\F}$ such that
\begin{equation}\label{estimeeX2}
\Norm{X}{k+2,2} -  n c_{1} \, |\tau| \, \Norm{N}{k+1,2} \infeg \tfrac{c_{1}}{4} \, \Norm{D\phigpis_{2}.\xi}{k+1,2} + \epsilon \Norm{\xi}{k+2,2} + C \, \Norm{\xi}{0,2}.
\end{equation}
We can choose a small positive constant $\epsilon_{0}$ so that $(\ref{estimeeN2}) + \epsilon_{0} (\ref{estimeeX2})$ implies  (\ref{35cg}).$\cqfd$\\[.4cm]

As a preliminary result, we shall establish a lemma corresponding to  the Time-symmetric version of proposition \ref{proplipcg} given later.

\begin{lem}{} Let $k\in\mathbb N$ such that $k+2>\frac n2$.
The operator $$D \phibf (g,0)^{*} \, (.,0) : \sob{k+2}{2}{}(\m) \longrightarrow \sob{k}{2}{} (\w{\S}) \;$$ is bounded and depends on $g$ in a Lipschitz way,
\begin{equation}\label{lip}
\bNorm{\big \lbrack D \phibf (g,0)^{*} - D \phibf (\wg,0)^{*} \big \rbrack \, (N,0)}{k,2} \infeg C \Norm{g - \tilde{g}}{\F} \; \Norm{N}{k+2,2} \, ,
\end{equation}
where the constant $C$ depends on $\gz  , \Norm{g}{\F}$ and $\Norm{\tilde{g}}{\F}$.
\end{lem}

\textbf{Proof}: Let us recall the statement of $D \phibf (g,0)^{*}$ :
\begin{equation}\label{adjphio}
D \phibf (g,0)^{*}. (N,0) = \lbrack \n_{i} \n_{j} N - g_{ij}\laplag N - \lbrack R_{ij} - \tfrac{1}{2}  (R(g) - 2 \Lambda) g_{ij} \rbrack N \rbrack \sqrt{g}.
\end{equation}
We begin by showing $D \phibf (g,0)^{*} \;$ is bounded.
Let us define the operator acting on functions
\begin{equation}\label{defO} 
O(N)= \n^{2}N - g \, \laplag N
\end{equation}
and note that $O(N)= L (\n^{2}N)$ where $L$ is a linear invertible operator. Thus
\begin{eqnarray*}
\Norm{O(N)}{k,2} &\infeg& c \, \Norm{\n^{2}N}{k,2} \infeg  c \, \left(\Norm{\nz^{2}N}{k,2} + \Norm{A\, d N}{k,2} \right)\\
& \infeg & C \, \Norm{N}{k+2,2},
\end{eqnarray*}
indeed, $A\, d N \simeq g^{-1} \nz g \; d N$ and using H\"older inequality (\ref{h1}) , $(\ref{19gk})$ and Sobolev inclusion,
\begin{eqnarray}\label{adnbis}
\Norm{A\, d N}{k,2} & \infeg& \Norm{g^{-1}}{k+2,2} \Norm{\nz g \, d N }{k,2} \nonumber\\
& \infeg& c \, \Norm{\nz g}{k+1,2} \Norm{d N}{k+1,2} \nonumber\\
& \infeg& C \, \Norm{N}{k+2,2}.
\end{eqnarray}
We go on with
\begin{eqnarray*}\label{dphiog}
\Norm{ D \phibf (g,0)^{*}. (N,0) \diagup \sqrt{g}}{k,2} &\infeg& \Norm{O(N)}{k,2} +  \Norm{\left(Ric \, g - Ric \, \gz \right) N}{k,2} + \Norm{(n-1) \gz N}{k,2} \nonumber\\
&& + \Norm{\left\lbrack Ric \, \gz + (n-1) \gz \right\rbrack N}{k,2} + \Norm{\tdemi n(n-1)\tau^{2}\,g \, N}{k,2}\\
&& + \tdemi \bNorm{\left\lbrack R(g) - 2 \Lambda + n(n-1)\tau^{2} \right \rbrack\,g \, N}{k,2}.
\end{eqnarray*}
Considering (\ref{integdifric}), (\ref{integriccg}),  we have
\begin{eqnarray*}
\Norm{(Ric \, g - Ric \, \gz) N}{k,2} &\infeg& C \,\Norm{N}{k+2,2}\\
\Norm{\lbrack Ric \, \gz + (n-1) \gz \rbrack N}{k,2}  &\infeg& c \, \Norm{N}{k+2,2}.
\end{eqnarray*}
For the scalar curvature term, using the lemma \ref{lemmePiotr} together with (\ref{integscal}) and (\ref{integdifscal}),
\begin{eqnarray*}
\Norm{(R(g) - 2 \Lambda + n(n-1)\tau^{2} )\,g\, N}{k,2} &\infeg& \Norm{(R(\gz) - 2 \Lambda + n(n-1)\tau^{2} )\,g\, N}{k,2}\\
&&+ \Norm{(R(g) - R(\gz) )\,g\, N}{k,2}\\
&\infeg& C \, \Norm{N}{k+2,2}.
\end{eqnarray*}
Now we have of course
$\Norm{(n-1) \gz N}{k,2} \infeg c \, \Norm{N}{k+2,2},$
and similarly,
$\Norm{\tdemi n(n-1)\tau^{2} g N}{k,2} \infeg c \, \Norm{N}{k+2,2}.$
We end up with
\begin{equation*}
\Norm{ D \phibf (g,0)^{*}. (N,0) \diagup \sqrt{g}}{k,2} \infeg C \, \Norm{N}{k+2,2},
\end{equation*}
and finally using again the lemma \ref{lemmePiotr}
\begin{equation}\label{majdphio}
\Norm{D \phibf (g,0)^{*}. (N,0)}{k,2} \infeg C \, \Norm{\sqrt{g}}{k+2,2} \Norm{N}{k+2,2} \infeg C \, \Norm{N}{2,2} ,
\end{equation}
where $C$ is a constant depending upon $\gz $ and $\Norm{g}{\F}$. The boundedness of the map is then proved.  

We now proceed to the  proof of equation (\ref{lip}). Let us denote respectively by
$\wn \,, \wlapla \, , Ric (\wg)$ and $R(\wg)$ the Levi-Civita connection, the Laplacian, the Ricci tensor and the scalar curvature of the Riemannian metric $\wg$.
In order to lighten notations, we also set $$D\phiogs N := D \phibf (g,0)^{*} \, (N,0) \; \; \mbox{and} \; \; D \phio (\wg)^{*} N := D \phibf (\wg,0)^{*} \, (N,0).$$
We split
\begin{equation*}
\lbrack D\phiogs - D \phio (\wg)^{*}\rbrack N = (\sqrt{g} - \sqrt{\w{g}}) \, \frac{D\phiogs N}{\sqrt{g}} + \sqrt{\w{g}} \left\lbrack \frac{D\phiogs N}{\sqrt{g}} - \frac{D\phio (\wg)^{*} N}{\sqrt{\wg}} \right\rbrack.
\end{equation*}
It directly implies
\begin{eqnarray}\label{declip}
\bNorm{\lbrack D\phiogs - D\phio (\tilde{g})^{*}\rbrack N}{k,2} 
&\infeg& \Norm{g -\w{g}}{\F} \BNorm{\frac{D\phiogs N}{\sqrt{g}}}{k,2} \nonumber\\
&&+ c \, \BNorm{\frac{D\phiogs N}{\sqrt{g}} - \frac{D\phio (\wg)^{*} N}{\sqrt{\wg}}}{k,2}.
\end{eqnarray}
Now, because
\begin{eqnarray*}
\Big( \frac{D\phiogs N}{\sqrt{g}} - \frac{D\phio (\wg)^{*} N}{\sqrt{\wg}}\Big) &=& (\n - \wn) \, d N + g\, \laplag N - \wg \, \wlapla N - \lbrack Ric (g) - Ric (\wg)\rbrack N \\
&&+ \tdemi \left\lbrack (R(g) - 2\Lambda)g - (R(\wg) - 2\Lambda) \wg \right\rbrack N,
\end{eqnarray*}
we have
\begin{eqnarray*}
\BNorm{\frac{D\phiogs N}{\sqrt{g}} - \frac{D\phio (\wg)^{*} N}{\sqrt{\wg}} }{k,2} &\infeg& \Norm{(\n - \wn) \, d N}{k,2} + \Norm{g \, \laplag N - \wg \, \wlapla N}{k,2}\\
&& + \tdemi \bNorm{\left\lbrack (R(g) - 2\Lambda) g - (R(\wg) - 2\Lambda) \wg \right\rbrack N}{k,2}\\
&& - \Norm{\lbrack Ric (g) - Ric (\wg)\rbrack N}{k,2}.
\end{eqnarray*}
We will estimate each  terms of the right hand side of the above inequality. 
\begin{list}{$\bullet$}
\item For the Hessians term, we write
\begin{equation}\label{hess}
\n - \wn = (g^{-1} - \wg^{-1}) \nz g + \wg^{-1} \nz (g - \wg).
\end{equation}
Using  (\ref{19gk}), we obtain
\begin{equation}\label{majorhess}
\Norm{(\n - \wn) \, d N}{k,2} 
 \infeg C \, \Norm{g - \wg}{k+2,2} \Norm{N}{k+2,2}.
\end{equation}
\item \item For the Laplacian terms, we decompose
\begin{eqnarray*}
g\, \laplag N - \wg \, \wlapla N &=& g\, \laplag N - \wg \, \laplag N + \wg \, \laplag N - \wg \, \wlapla N\\
&=& (g - \wg ) \laplag N + \wg (\laplag N - \wlapla N)\\
&=& (g - \wg ) g^{-1} \n d N + \wg  (g^{-1} - \wg^{-1}) \n d N + \wg \wg^{-1}(\n - \wn) d N.
\end{eqnarray*}
Using   (\ref{19gk}), we deduce
\begin{eqnarray*}
\Norm{g\,\laplag N - \wg\, \wlapla N}{k,2}
& \infeg& c \, \Norm{g - \wg}{k+2,2} \Norm{\n d N}{k,2} + c \, \Norm{(\n - \wn) d N}{k,2}.
\end{eqnarray*}
Now, considering (\ref{adnbis}) and given that $\n \simeq A + \nz$ , we have
\begin{equation}\label{hessn}
\Norm{\n d N}{k,2} \infeg C \, \Norm{N}{k+2,2} \,,
\end{equation}
Using $(\ref{majorhess})$ and $(\ref{hessn})$ , we  finally get
\begin{equation*}
\Norm{g\, \laplag N - \wg\, \wlapla N}{k,2} \infeg C \, \Norm{g - \wg}{k+2,2} \Norm{N}{k+2,2}.
\end{equation*}
\item For the Ricci tensors term, we
 define $\wade{i}{k}{j} = \wcde{i}{k}{j}-\cdez{i}{k}{j}$ and we
set $$\w T := \wn - \n = \wg^{-1} \nz \wg - g^{-1} \nz g = (g^{-1} - \wg^{-1}) \nz g + \wg^{-1} \nz (g - \wg).$$
Using  (\ref{19gk}),
\begin{eqnarray}\label{normT}
\Norm{\w T}{k+1,2} 
& \infeg& C \Norm{g - \wg}{k+2,2}.
\end{eqnarray}
We can show, adding and subtracting $Ric(\gz)$ and using $(\ref{riccicov})$, that
$$\lbrack Ric (g) - Ric (\wg)\rbrack N \simeq (\nz \w T + \w{A}\w T + \w T^{2}) N \, ,$$
which leads to
\begin{equation}\label{majric}
\Norm{\lbrack Ric (g) - Ric (\wg)\rbrack N}{k,2} \infeg \Norm{\nz \w T N}{k,2} + \Norm{\w{A}\w T N }{k,2} + \Norm{\w T^{2} N }{k,2}.
\end{equation}
Using  (\ref{19gk}) and (\ref{normT})
\begin{eqnarray*}
\Norm{\nz \w T N}{k,2} 
&\infeg& C \, \Norm{g - \wg}{k+2,2} \Norm{N }{k+2,2}
\end{eqnarray*}
The same method for the term $\w{A}\w T N $ gives , considering (\ref{19gk})
\begin{eqnarray*}
\Norm{\w{A}\w T N }{k,2}
&\infeg& C \, \Norm{g - \wg}{k+2,2} \Norm{N }{k+2,2}.
\end{eqnarray*}
In the same way  with $(\ref{normT})$ ,
\begin{eqnarray*}
\Norm{\w T^{2} N }{k,2} 
&\infeg& C \, \Norm{g - \wg}{k+2,2}^{2} \Norm{N }{k+2,2}.
\end{eqnarray*}
Replacing in (\ref{majric}), we obtain
\begin{equation}\label{majric2}
\Norm{\lbrack Ric (g) - Ric (\wg)\rbrack N}{k,2} \infeg C\, \Norm{g - \wg}{k+2,2} \Norm{N}{k+2,2}.
\end{equation}
\item For the scalar curvatures term,  we write
\begin{eqnarray*}
(R(g)- 2\Lambda) g\,- (R(\wg)- 2\Lambda) \wg &=& (g - \wg)(R(g)- 2\Lambda) + \wg \wg^{-1}(Ric \, g - Ric \, \wg)\\
&&+ \wg (g^{-1} - \wg^{-1}) Ric \, g\\
&=& (g - \wg) \left\lbrack R(g)- 2\Lambda + n(n-1)\tau^{2} \right\rbrack\\
&&-n(n-1)\tau^{2}(g - \wg) + \wg \wg^{-1}(Ric \, g - Ric \, \wg) \\
&&+ \wg (g^{-1} - \wg^{-1}) \big\{ Ric \, g - Ric \, \gz \big\} .
\end{eqnarray*}
The inequilty  (\ref{19gk}) and (\ref{majric2}) will yield
\begin{equation*}
\Norm{\lbrack (R(g)- 2\Lambda) g\,- (R(\wg)- 2\Lambda) \wg \rbrack N}{k,2} \infeg C \, \Norm{g - \wg}{k+2,2}\Norm{N}{k+2,2} \, ,
\end{equation*}
because for instance $\forall u \in\sob{k+2}{2}{}, \forall v \in \sob{k}{2}{}$ ,
\begin{eqnarray*}
\Norm{(g - \wg ) \, u \, v \, N}{k,2}  &\infeg& \Norm{g - \wg}{k+2,2} \Norm{u}{k+2,2} \Norm{v N}{k,2}\\
&\infeg& C \, \Norm{g - \wg}{k+2,2} \Norm{u}{k+2,2} \Norm{v}{k,2} \Norm{N}{k+2,2} ,
\end{eqnarray*}
where $C$ is a positive constant depending on $\gz $ and $\Norm{g}{\F}$.
\end{list}
Putting the pieces all together in (\ref{declip}) and taking (\ref{majdphio}) into account lead to
\begin{equation*}
\Norm{\lbrack D\phiogs - D\phio (\tilde{g})^{*}\rbrack N}{k,2} \infeg C \, \Norm{g - \wg}{k+2,2}\Norm{N}{k+2,2}, 
\end{equation*}
and close the proof. $\cqfd$\\[.3cm]
The dependence in $(g,\pi)$ of $P^*$ is controlled as follows :
\begin{prop}{}\label{proplipcg}
When $k+2>n/2$, the operator $P^{*} : \sob{k+2}{2}{}(\T) \longrightarrow  \sob{k}{2}{} \;$ is bounded and satisfies
\begin{equation}\label{estimeePs}
\Norm{\xi}{k+2,2} \infeg c \, \Norm{P^{*} \xi }{k,2} + C \, \Norm{\xi}{0,2} \, ,
\end{equation}
where $C$ depends on $\gz$ and $\Norm{(g,\pi)}{\F}$.\\
Moreover, $\Pgpis$ depends on $(g,\pi) \in \F$ in a Lipschitz way,
\begin{equation}\label{lipcg}
\Norm{(\Pgpis - \Pgpisw) \, \xi}{k,2} \infeg C_{1} \Norm{(g - \t{g}, \pi - \t{\pi})}{\F} \; \Norm{\xi}{k+2,2}\, ,
\end{equation}
where constant $C_{1}$ depends on $\gz  , \Norm{(g,\pi)}{\F}$ and $\Norm{(\t{g},\t{\pi})}{\F}$.
\end{prop}
\textbf{Proof}: Let us begin by showing $P^{*}$ is bounded, \ie
\begin{equation}\label{majPs}
\Norm{P^{*} \, \xi}{k,2} \infeg C \; \Norm{\xi}{k+2,2}.
\end{equation}
We set
$$
\begin{cases}
\Ps = \Pgpis,\\
D\phisun = D\phigpis_{1},\\
D\phisdeux = D\phigpis_{2}.
\end{cases}
$$
From (\ref{defPs}), we have
\begin{eqnarray}\label{majPsbis}
\Norm{P^{*} \, \xi}{k,2} &\infeg& c \; (\Norm{D\phisun.\xi}{2} + \Norm{\n D\phisdeux.\xi}{k,2}) \nonumber \\
 &\infeg& c \; (\Norm{D\phisun.\xi}{k,2} + \Norm{\nz D\phisdeux.\xi}{k,2} + \Norm{A D\phisdeux.\xi}{k,2}).
\end{eqnarray}
From (\ref{dphis1}), (\ref{19gk}),  we first estimate
\begin{eqnarray}\label{majdphiun}
\Norm{D\phisun.\xi}{k,2} &\infeg& c \, \big(\Norm{T}{k,2} + \Norm{\Sz (X)}{k,2} + \Norm{N}{k,2} \big) + C \, \big(\Norm{\xi}{k+2,2} + \Norm{\nz X}{k+1,2} \big) \nonumber \\
&\infeg& c \, \big(\Norm{N}{k+2,2} + \Norm{X}{k+1,2} + \Norm{N}{k,2} \big) + C \, \big(\Norm{\xi}{k+2,2} + \Norm{\nz X}{k+1,2} \big) \nonumber\\
&\infeg& C \; \Norm{\xi}{k+2,2}.
\end{eqnarray}
From (\ref{dphis2}) along with (\ref{h1}), (\ref{19gk}),  we can also control
\begin{equation}\label{majdphideux}
\Norm{D\phisdeux.\xi}{k,2} \infeg c \, \big(\Norm{\Sz (X)}{k,2} + \Norm{N}{k,2} \big) + C \, \Norm{\xi}{k+1,2},
\end{equation}
\begin{eqnarray*}
\Norm{A D\phisdeux.\xi}{k,2} &\infeg&  c \, \big( \Norm{A \Sz (X)}{k,2} + \Norm{A N}{k,2} \big) + \Norm{\xi}{k+2,2} \\
&\infeg&  C \, \big( \Norm{\Sz (X)}{k+1,2} + \Norm{N}{k+1,2} \big) + \Norm{\xi}{k+2,2} ,\\
\Norm{\nz D\phisdeux.\xi}{k,2} &\infeg& c \, \big(\Norm{\nz \Sz (X)}{k,2} + \Norm{ \nz N}{k,2} \big) + \Norm{\xi}{k+2,2}\\
&\infeg&  C \, \Norm{\xi}{k+2,2}.
\end{eqnarray*}
Consequently,
\begin{equation}\label{majdphideux12}
\Norm{D\phisdeux.\xi}{k,2} \infeg \Norm{D\phisdeux.\xi}{k+1,2} \infeg C \, \Norm{\xi}{k+2,2}.
\end{equation}
Every term of the right hand side of  (\ref{majPsbis}) is then dominated by $\Norm{\xi}{k+2,2}$ leading to (\ref{majPs}).
 The estimate (\ref{estimeePs}) satisfied by $P^{*}$ directly comes from (\ref{35cg}).
We now look into the Lipschitz behaviour of $P^{*}$. We 
set
$$
\begin{cases}
\Psw = \Pgpisw,\\
D\phiwsun = D\phigpiws_{1},\\
D\phiwsdeux = D\phigpiws_{2}.
\end{cases}
$$
Let us write
\begin{eqnarray*}
(\Ps - \Psw) \, \xi &=&
\left(
\begin{array}{c}
g^{-1 \slash 4} \, D\phisun.\xi - \t{g}^{-1 \slash 4} \, D\phiwsun .\xi \\
 g^{1 \slash 4} \, \n D\phisdeux.\xi - \t{g}^{1 \slash 4} \, \nw D\phiwsdeux.\xi 
\end{array}
\right)
=:
\left(
\begin{array}{c}
E \\
 F
\end{array}
\right),
\end{eqnarray*}
so
\begin{equation}\label{majlipEF}
\Norm{(\Ps - \Psw) \, \xi}{k,2} \infeg \Norm{E}{k,2} + \Norm{F}{k,2}.
\end{equation}
We start to estimate
\begin{eqnarray*}
E &=& g^{-1 \slash 4} \, D\phisun.\xi - \t{g}^{-1 \slash 4} \, D\phiwsun .\xi \\
&=& (g^{-1 \slash 4} - \t{g}^{-1 \slash 4}) \, D\phisun.\xi + \t{g}^{-1 \slash 4} \, ( D\phisun.\xi -D\phiwsun .\xi ).
\end{eqnarray*}
Using  (\ref{19gk}),
\begin{eqnarray*}
\Norm{E}{k,2} &\infeg& \Norm{(g^{-1 \slash 4} - \t{g}^{-1 \slash 4}) \, D\phisun.\xi}{k,2} + \Norm{\t{g}^{-1 \slash 4} \, ( D\phisun.\xi -D\phiwsun .\xi )}{k,2}\\
&\infeg& c \, \Norm{g - \t{g}}{\F}\Norm{D\phisun.\xi}{k,2} + c \, \Norm{D\phisun.\xi -D\phiwsun .\xi}{k,2}.
\end{eqnarray*}
From (\ref{expdphi1dphi0}), we expand
\begin{eqnarray*}
D\phisun.\xi -D\phiwsun .\xi &=& \lbrack D \phibf (g,0)^{*} - D \phibf (\wg,0)^{*} \rbrack \, (N,0) + (\Pi\sqrt{g} - \t{\Pi} \sqrt{\t{g}}) N  +  X \nz (\pi - \t{\pi}) \\
&& + (\pi - \t{\pi}) \nz X + A X (\pi - \t{\pi}).
\end{eqnarray*}
Using (\ref{h1}), (\ref{19gk}), 
we already have
\begin{eqnarray*}
\Norm{(\pi - \t{\pi}) \nz X}{k,2} + \Norm{X \nz (\pi - \t{\pi})}{k,2} &\infeg& \Norm{\pi - \t{\pi}}{k+1,2}\Norm{\nz X}{k+1,2}\\
&& + \Norm{\nz (\pi - \t{\pi})}{k,2} \Norm{X}{k+2,2}\\
&\infeg& c \, \Norm{\pi - \t{\pi}}{k+1,2} \; \Norm{X}{k+2,2},
\end{eqnarray*}
and
\begin{eqnarray*}
\Norm{A X (\pi - \t{\pi})}{k,2} &\infeg& \Norm{A (\pi - \t{\pi})}{k,2} \Norm{X}{k+2,2}\\
&\infeg& C \, \Norm{\pi - \t{\pi}}{k+1,2} \; \Norm{X}{k+2,2}.
\end{eqnarray*}
Now, we write formally 
\begin{eqnarray*}
\Pi\sqrt{g} - \t{\Pi} \sqrt{\t{g}}&\thicksim& \frac{1}{\sqrt{g}}  g^{-1} \pi^{2}  - \frac{1}{\sqrt{\t{g}}} \t{g}^{-1} \t{\pi}^{2}\\
&\thicksim& \frac{1}{\sqrt{g}}  (g^{-1} -  \t{g}^{-1} ) \pi^{2} + \frac{1}{\sqrt{g}} \t{g}^{-1} (\pi^{2} - \t{\pi}^{2}) + (\frac{1}{\sqrt{g}} - \frac{1}{\sqrt{\t{g}}}) \t{g}^{-1} \t{\pi}^{2},
\end{eqnarray*}
leading to
\begin{eqnarray*}
\Norm{(\Pi\sqrt{g} - \t{\Pi} \sqrt{\t{g}})N}{k,2}
&\infeg& C \, \Norm{(g - \t{g}, \pi - \t{\pi})}{\F} \Norm{N}{k+2,2}.
\end{eqnarray*}
Given also (\ref{lip}), we obtain
\begin{equation*}
\Norm{D\phisun.\xi -D\phiwsun .\xi}{k,2} \infeg C \, \Norm{(g - \t{g}, \pi - \t{\pi})}{\F} \Norm{\xi}{k+2,2},
\end{equation*}
and taking (\ref{majdphiun}) into account,
\begin{equation}\label{majE}
\Norm{E}{k,2} \infeg C \, \Norm{(g - \t{g}, \pi - \t{\pi})}{\F} \Norm{\xi}{k+2,2}.
\end{equation}

We will now estimate the term
\begin{eqnarray*}
F &=& g^{1 \slash 4} \, \n D\phisdeux.\xi - \t{g}^{1 \slash 4} \, \nw D\phiwsdeux.\xi \\
&=& g^{1 \slash 4} \, (\n - \nw) D\phisdeux.\xi + (g^{1 \slash 4} - \t{g}^{1 \slash 4}) \, \nw D\phiwsdeux.\xi +  g^{1 \slash 4} \, \nw (D\phisdeux.\xi - D\phiwsdeux.\xi).
\end{eqnarray*}
Using (\ref{hess}),(\ref{h1}), (\ref{19gk}), 
\begin{eqnarray*}
\Norm{F}{k,2} &\infeg& c \, \Norm{\n - \nw}{k+1,2}\Norm{D\phisdeux.\xi}{k+1,2} + \Norm{g^{1 \slash 4} - \t{g}^{1 \slash 4}}{k+2,2}\Norm{ \nw D\phiwsdeux.\xi}{k,2}\\
&&+ c \, \Norm{\nz (D\phisdeux.\xi - D\phiwsdeux.\xi)}{k,2} + c \, \Norm{A (D\phisdeux.\xi - D\phiwsdeux.\xi)}{k,2}\\
&\infeg& C \, \Norm{g - \t{g}}{\F} \Norm{D\phisdeux.\xi}{k+1,2} +  c \, \Norm{g -  \t{g}}{\F}\Norm{ \nw D\phiwsdeux.\xi}{k,2}\\
&&+ c \, \Norm{\nz (D\phisdeux.\xi - D\phiwsdeux.\xi)}{k,2} + c \, \Norm{A (D\phisdeux.\xi - D\phiwsdeux.\xi)}{k,2}.
\end{eqnarray*}
Considering (\ref{majdphideux}) and (\ref{majdphideux12}), one has
\begin{eqnarray}\label{majFbis}
\Norm{F}{k,2} &\infeg& C \, \Norm{g - \t{g}}{\F} \Norm{\xi}{k+2,2} + c \, \Norm{\nz (D\phisdeux.\xi - D\phiwsdeux.\xi)}{k,2}\\
&& + c \, \Norm{A (D\phisdeux.\xi - D\phiwsdeux.\xi)}{k,2}, \nonumber
\end{eqnarray}
with, formally,
\begin{eqnarray*}
D\phisdeux.\xi - D\phiwsdeux.\xi &\thicksim& (K - \t{K}) N + ( A - \t{A})X \\
&\thicksim& (\pi - \t{\pi}) N + ( \n - \nw)X .
\end{eqnarray*}
Using (\ref{hess}),(\ref{h1}), (\ref{19gk}), we deduce
\begin{eqnarray*}
\Norm{\nz (D\phisdeux.\xi - D\phiwsdeux.\xi)}{k,2} &\infeg& c \, \Norm{\nz (\pi - \t{\pi})}{k,2}\Norm{N}{k+2,2} +  c \, \Norm{\pi - \t{\pi}}{k+1,2}\Norm{\nz N}{k+1,2}\\
&& + \Norm{\nz(\n - \nw)}{k,2}\Norm{X}{k+2,2} + \Norm{\n - \nw}{k+1,2}\Norm{\nz X}{k+1,2}\\
&\infeg&  c \, \Norm{\pi - \t{\pi}}{k+1,2}\Norm{N}{k+2,2} + c \, \Norm{\n - \nw}{k+1,2}\Norm{X}{k+2,2} \\
&\infeg& C \, \Norm{(g - \t{g},\pi - \t{\pi})}{\F} \Norm{\xi}{k+2,2}.
\end{eqnarray*}
In the same way,
\begin{eqnarray*}
\Norm{A (D\phisdeux.\xi - D\phiwsdeux.\xi)}{k,2} &\infeg& c \, \Norm{A(\pi - \t{\pi})}{k,2}\Norm{N}{k+2,2} + \Norm{A(\n - \nw)}{k,2}\Norm{X}{k+2,2}\\
&\infeg& C \, \Norm{(g - \t{g},\pi - \t{\pi})}{\F} \Norm{\xi}{k+2,2}.
\end{eqnarray*}
We deduce from (\ref{majFbis})
\begin{equation}\label{majF}
\Norm{F}{2} \infeg C \, \Norm{(g - \t{g}, \pi - \t{\pi})}{\F} \Norm{\xi}{k+2,2}.
\end{equation}
  {The  desired Lipschitz estimate}  (\ref{lipcg}) arises from (\ref{majlipEF}), considering (\ref{majE}) and (\ref{majF}).$\cqfd$\\[.3cm]

We claim that  the estimate  {(\ref{35cg}) of Corollary \ref{corestiadjL2}} is also satisfied by weak solutions $\xi$ only in $\sob{-k}{2}{} (\T)$. More precisely, we say that $\xi \in \L$ is a weak solution of $D \phigpis \xi = (f_{1} , f_{2}) \;,$ with $(f_{1} , f_{2}) \in \sob{k}{2}{}(\w{\S}) \times \sob{k+1}{2}{}(\S)$ when
$$ \langle \xi, D \phigpi .(h,p) \rangle= \int_{\m} \langle (f_{1} , f_{2}) , (h,p) \rangle_{\gz} \; , \; \forall (h,p) \in  
\sob{k+2}{2}{} (\S)\times\sob{k+1}{2}{} (\w{\S}).$$
Note that it suffices to verify  the equality for any  $(h,p) \in \cinf_{c} (S) \times \cinf_c( \w{\S})$ by density.
\begin{prop}\label{propRegul}
Let  $k+2>\frac n2$ and   $(g,\pi) \in \G^{+} \times \K \, , $ $(f_{1} , f_{2}) \in \sob{k}{2}{}(\w{\S}) \times \sob{k+1}{2}{}(\S)$. Assume that  $\xi \in \L$ is a weak solution of $D \phigpis \xi = (f_{1} , f_{2})$, then $\xi \in \sob{k+2}{2}{}(\T)$ is a strong solution and satisfies  {(\ref{35cg})}.
\end{prop}
\textbf{Proof}: 
We start with an adaptation to the dimension $n$ of the beginning of the proof of the proposition 3.5 of \cite{Bartnik2005} .
More precisely, the equation  $P^*\xi=f$ can be rewritten in local coordinate 
$$
A.\partial^2\xi+B.\partial\xi+C.\xi=f, 
$$ with $A\in \sob{k+2}{2}{} $ invertible, $B\in \sob{k+1}{2}{} $, $C\in  \sob{k}{2}{}  $, $\xi\in\sob{-k}{2}{}.$
This is equivalent to an equation of the form (see equation (39) in \cite{Bartnik2005} )
\begin{equation}\label{B39}
\partial^2\xi+\partial(b\xi)+c\,\xi=f,
\end{equation}
with $b\in W^{k+1,2}$ and $c,f\in W^{k,2}$.

From  now one  the proof is different from the Bartnik's one, we do not take a trace of the equation. Let $p\in\mathbb N$, $p>0$ such  that  $k+2-\frac n2>\frac1p$, 
by multiplications in Sobolev spaces  (see \cite{HNT2009}  lemma 28 for instance) we infer that  $c\,\xi\in W^{-k-2+\frac1p,2}$ and $b\,\xi\in W^{-k-1+\frac1p,2}$.
This proves that $\partial^2\xi\in W^{-k-2+\frac1p,2}$ by equation (\ref{B39}), so combined with $\partial\xi\in W^{-k-1,2}\subset W^{-k-2+\frac1p,2}$ implies  $\partial\xi\in W^{-k-1+\frac1p,2}$.
Again this last fact added to $\xi\in W^{-k,2}\subset W^{-k-1+\frac1p,2}$ leads to  $\xi\in W^{-k+\frac1p,2}$. We have then improved the regularity of  $\xi$.
Bootstrapping in this way  (we can add $1/p$ to the regularity at each step)  we will end up with  $\xi\in W^{k+2,2}$.
$\cqfd$\\
In the compact manifold setting, we  may assume  that the  (weak) kernel of $D\phigpi^*$ is trivial and forget this section. But either for a practical verification and for an adaptation to non  compact manifold, where one has to prove there is no kernel under appropriate behaviour,  this  regularity result is very important.

\section{The submanifold structure}\label{SectionStructure}
{We are ready to provide  the smooth Hilbert submanifold structure  of the set of solutions to the vacuum constraint equations.\\
We start with a well known fact.}
\begin{lem}{}\label{lemiso}
Let $X,Y$ be two Banach spaces and $T$ a linear operator with closed range.
\begin{eqnarray*}
T: X &\rightarrow &Y\\
T^{*}: Y^{*}& \rightarrow &X^{*}
\end{eqnarray*}
then $({\Coker} T)^{*} \simeq {\ker}\, T^{*}$,
where ${\Coker} \,T := \faktor{Y}{\Im T}$ is a Banach space.
\end{lem}
\textbf{Proof}:  
This is a classical argument, see \cite{DF2016} lemma 22 for instance
$\cqfd$\\

{We can now state our main result.}
\begin{thm}{}\label{mainthmmanifold}
Let $\phibf : \F \rightarrow \L^{*}$ be the constraint operator and assume that $k+2>\frac n2$.
For every $\epsilon \in \L^{*}$ , the set of solutions of the constraint equations without KID's
$$\contrainte (\epsilon) := \lbrace (g,\pi) \in \F : \ker D\phigpi^*=\{0\},\;\phigpi = \epsilon \rbrace$$
is a submanifold of $\F$. In particular,  the space of solutions, without KID's, of the vacuum constraint equations  $\contrainte = \contrainte (0)$ has a Hilbert submanifold structure.
\end{thm}
  {In order to prove the Theorem 
\ref{mainthmmanifold},we need  to show:}
\begin{list}{$\bullet$}
\item $\ker\; D \phigpi$ splits.
\item \item $D \phigpi$ is surjective.
\end{list}
\medskip
$D \phigpi$ being a bounded operator, its kernel is closed by continuity {and 
$T_{(g,\pi)}\F=\sob{k+2}{2}{} (\S)\times\sob{k+1}{2}{} (\w{\S})$ can be written as a direct sum of $\ker D \phigpi$ and its orthogonal complement $(\ker D \phigpi)^{\perp}$, which is always closed. Hence $\ker D \phigpi$ splits.}\\
 {The assumption about  triviality of ker $D \phigpis$, leads to}
\begin{equation*}
\left( {\ker} D \phigpis \right)^{\perp} = \L^{*}.
\end{equation*}
Using the classical relation
\begin{equation*}
\left( {\ker} D \phigpis \right)^{\perp} = \overline{{\Im} D \phigpi} ,
\end{equation*}
we get
$$\overline{{\Im} D \phigpi} = \L^{*}.$$
Thus, in order to obtain the surjectivity of  $D \phigpi$ , we will  prove it has closed range. 
For that it suffice to prove the range is  the direct sum of a closed space and a  finite dimensional space.
To do so, we consider particular variations $(h,p)$ of $(g,\pi)$ of the form 
\begin{equation}\label{defhp}
\begin{cases}
h_{ij} = 2 \, y \, g_{ij}\\
p^{ij} = \big(2 \S(Y)^{ij} - g^{ij} \, {\tr}_{g}\S(Y) -(n-1)(n-2) \tau y \, g^{ij} \big) \sqrt{g}
\end{cases}
,
\end{equation}
determined from fields $(y,Y^{i})$. We define the operator
\begin{equation}
F(y, Y^{i}) = \lbrack F_{0}(y, Y^{i}),F_{i}(y, Y^{i})\rbrack = \lbrack D \phiogpi (h,p),  D \phiigpi (h,p) \rbrack.
\end{equation}
The equations (\ref{def0}) and (\ref{defi}) provide, \footnote{here a fixed constant $\kappa$ can be chosen to be zero for a compact manifold but we want to make the calculation easily adaptable to open manifold asymptotic to Einstein models, see section \ref{noncomp}}
\begin{equation*}
\left\lbrace
\begin{array}{rcl}
 F_{0}(y, Y^{i})& =& 2(n-1) \sqrt{g} \, \lbrack -\Delta y -\kk n y  \rbrack + (4-n) \, \phiogpi \, y + 2(n-2)\tau\,{  {\div }}Y \sqrt{g}\\
&& + \sob{k}{2}{} \, \lbrack y + Y \rbrack + \sob{k+1}{2}{} \, \nz Y,\\[.2cm]
F_{i}(y, Y^{i}) &=& -2 \sqrt{g} \, \lbrack -\Delta Y^{i} -\kk (n-1) Y^{i}  \rbrack + 2 \, \phiigpi \, y + \sob{k+1}{2}{} \, \nz y +\sob{k}{2}{} \, \lbrack y + Y \rbrack.
\end{array}
\right.
\end{equation*}
{In order to prove Fredholm properties of $F$, we compare to the corresponding one related to $\gz$}
\begin{defi}{}\label{defopasympt}
Let $k\in\mathbb N$. We say an operator $P$ of the form
\begin{equation*}
Pu = a^{ij}(x)\nz^{2}_{ij}u + b^{i}(x)\d_{i}u + c(x)u
\end{equation*}
is well related to $\laplaz$  if there exists $n<q(k+1)<\infty$, and two positive constants $C_{1} , \lambda$ such that
\begin{eqnarray*}
&&\lambda \norm{\xi}{\gz}^{2} \infeg a^{ij}(x) \xi_{i}\xi_{j} \infeg \lambda^{-1} \norm{\xi}{\gz}^{2} , \forall x \in \m \; , \; \xi \in T\m. \\
&&\Norm{a^{ij} - \gz^{ij}}{k+1,q} + \Norm{b^{i}}{k,q} + \Norm{c}{k,q\frac{k+1}{k+2} }\infeg C_{1}.
\end{eqnarray*}
\end{defi}
From the lemma \ref{lemmePiotr}, the operator $P$ is then bounded from $W^{k+2,2}$ to $W^{k,2}$.
{In our situation, we are interested in the Laplacian relative to the metric $g$.}
\begin{prop}{}\label{deltaasymptdelta0} Let $k\in\mathbb N$ such that $k+2>\frac n 2$. Let $c\in W^{k,2}$
and let $g \in \G^{+}$ . Then $ \Delta+c$ is well related to $\laplaz$.
\end{prop}
\textbf{Proof}: Recall that \begin{eqnarray}\label{expressionDelta}
\Delta = g^{ij} \n^{2}_{ij} &=& g^{ij} \nz^{2}_{ij} + g^{ij} (\n_{i} - \nz_{i}) \nz_{j} \nonumber\\
&=& g^{ij} \nz^{2}_{ij} - g^{ij} \ade{i}{k}{j} \, \nz_{k}.
\end{eqnarray}
The metrics $g$ and $\gz$ being equivalent, equation (\ref{holdcontcg}) directly gives
\begin{equation*}
\lambda \norm{\xi}{\gz}^{2} \infeg g^{ij}(x) \xi_{i}\xi_{j} \infeg \lambda^{-1} \norm{\xi}{\gz}^{2} , \forall x \in \m \; , \; \xi \in T\m.
\end{equation*}
Setting $$b^{k} = g^{ij} \ade{i}{k}{j},$$ then $b \in \sob{k+1}{2}{}$ from (\ref{A}).
Let us choose $q=2\frac{k+2}{k+1}$, so $(k+1)q>n$ and $2=q\frac{k+1}{k+2}$.
Given the Sobolev inequality, there exists a constant $C_{1} >0 $ such that 
\begin{equation*}
\Norm{g^{ij} - \gz^{ij}}{k+1,q} + \Norm{b^{k}}{k,q} +  \Norm{c}{k, q\frac{k+1}{k+2}} \infeg C \, (\Norm{g^{ij} - \gz^{ij}}{k+2,2} + \Norm{b^{k}}{k+1,2}+\Norm{c}{k,2}) \infeg C_{1}. \cqfd
\end{equation*}
The operator $\A = - \Delta -\kk n$, acting on functions, will be of great interest.
{It satisfies a  classical elliptic  estimate, valid for any weight $s$.}

\begin{prop}{}\label{propelliptique}
Let $g\in \Gplus$  and $\A = - \Delta -\kk n$. There exists a constant $C = C(n,p,q,s,C_{1},\lambda)$ such that if $u \in \leb{2}{}$ and $\A u \in \sob{k}{2}{}$, then $u \in \sob{k+2}{2}{}$  {and}
\begin{equation}\label{estiAu2}
\Norm{u}{k+2,2} \infeg C \, \left(\Norm{\A u }{k,2} + \Norm{u}{0,2} \right).
\end{equation}
\end{prop}
\textbf{Proof}: By elliptic regularity, $u \in \sob{2}{2}{loc}$ and the estimate arises from interior estimates (see \cite{GT}, \cite{HNT2009}, \cite{BH2017} for example) and partition of unity.$\cqfd$\\

We need a similar result for an operator acting on $1$-forms.
Let us define $\Bz = - \mathring\Delta -\kk( n-1)$, be a Laplacian  acting on $1$-forms.
\begin{thm}{}\label{thcB}
We assume $g\in\Gplus$.
Setting $B = - \Delta -\kk( n-1)$.\\
Then $B: \sob{k+2}{2}{}(T^{*}\m) \rightarrow \sob{k}{2}{}(\T^{*}\m)$ is bounded. Furthermore, it satisfies
\begin{equation}\label{sbeB}
\Norm{Y}{k+2,2} \infeg C \,\left(\Norm{BY}{k,2} + \Norm{Y}{0,2} \right).
\end{equation}
In particular, $B$ is a semi-Fredholm operator.
\end{thm}

We are now ready to prove Fredholm property of the operator $F$ associated to our special variations.
\begin{thm}{}\label{thc}
When $k+2>\frac n2$,  the operator$$F: {\sob{k+2}{2}{}(\m) \times \sob{k+2}{2}{}(\T\m)} \rightarrow \leb{2}{}(\T^{*} \otimes \Lambda^{3} T^{*}\m) :=\L^{*}$$ is bounded. Furthermore, it satisfies
\begin{equation}\label{sbedelta}
\Norm{(y,Y)}{k+2,2} \infeg C \, \left(\Norm{F(y,Y)}{k,2} + \Norm{(y,Y)}{0,2} \right).
\end{equation}
In particular, $F$ is a semi-Fredholm operator.
\end{thm}
\textbf{Proof}:
Starting from the definition of $F$, we find like before
$$\Norm{F(y,Y)}{k,2}\infeg C \, \Norm{(y,Y)}{k+2,2} ,$$
where $C$ is a constant depending on $\gz$ and $\Norm{g}{\F}$.\\
Hence $F$ is a bounded (continuous) operator. Plugging the expression of $F_{0}(y, Y^{i})$ in (\ref{estiAu2}) and using H\"older inequality (\ref{h1}) , lemma \ref{Ehrlingadapte}  (see remark \ref{produitadapte}) ,Ehrling inequality (\ref{Ehrling}) along with  and $\phiogpi\in \sob{k}{2}{}$,
\begin{eqnarray}\label{fredA}
\Norm{y}{k+2,2} &\infeg& C \, \left(\Norm{-\Delta y -\kk ny}{k,2} + \Norm{y}{0,2}\right) \nonumber\\
&\infeg& C \, \left(\Norm{F_{0}(y, Y)}{k,2} + \Norm{(y,Y)}{0,2}\right).
\end{eqnarray}
Plugging the expression of $F_{i}(y, Y^{i})$ in (\ref{sbeB}) and using H\"older inequality (\ref{h1}) , (\ref{uinfdk}) ,Ehrling inequality (\ref{Ehrling}) along with lemma \ref{lemmePiotr} and $\phiigpi\in \sob{k}{2}{}$,
\begin{eqnarray}\label{fredB}
\Norm{Y}{k+2,2} &\infeg& C \, \left(\Norm{-\Delta Y -\kk (n-1)Y}{k,2} + \Norm{Y}{0,2}\right) \nonumber\\
&\infeg& C \, \left(\Norm{F_{i}(y, Y)}{k,2} + \Norm{(y,Y)}{0,2} + \Norm{Y}{0,2}\right).
\end{eqnarray}
Finally, combination of (\ref{fredA}) and (\ref{fredB}) gives (\ref{sbedelta}). It is now standard to deduce from the estimate (\ref{sbedelta})  that $F$ is semi-Fredholm.$\cqfd$\\[.3cm]

We can now argue like in the proof of corollary \ref{corFredholm} to prove that $F$ is a Fredholm operator.
We approximate the metric $g$ by a smooth one $g_\epsilon$ to produce an operator $F_\epsilon$ close to $F$.
Now $F_\epsilon$ and its adjoint $F^{*}_\epsilon$ have similar structure
$$F^{*}_\epsilon: \sob{-k}{2}{}(\T)\rightarrow \sob{-k-2}{2}{}(\T^{*} \otimes \Lambda^{3}T^{*}\m).$$
Let $\w{F^{*}}_\epsilon$ be the restriction of $F^{*}_\epsilon$ defined as follows
$$\w{F^{*}}_\epsilon: \sob{k+2}{2}{}(\T) \rightarrow \sob{k}{2}{}(\T^{*} \otimes \Lambda^{3}T^{*}\m).$$
Booth  $F_\epsilon$ and $\w{F^{*}}_\epsilon$  satisfies an estimate like (\ref{sbedelta}) and because of elliptic regularity the kernel of $\w{F^{*}}_\epsilon$ is  the same as the kernel of ${F^{*}}_\epsilon$.
We conclude that  $F_\epsilon$ is Fredholm.  $F$  being semi-Fredhom and a limit  of Fredhom operators, it is Fredholm. Thus  Im$F$ is closed and its cokernel is finite dimensional. 

 {We can now close the proof 
of Theorem \ref{mainthmmanifold} by the following   argument.
}\\
 The space $\text{Coker}  F = \quotient{\L^{*}}{\text{Im} F}$ is finite  dimensional.
The operator $F$ satisfies
$$\text{Im} F \subset \text{Im} D \phigpi \subset \L^{*}.$$
Let $\pi$ be the canonical projection:
$$\pi : \L^{*} \rightarrow \quotient{\L^{*}}{\text{Im} F}.$$
$\pi (\text{Im} D \phigpi)$ is a subspace of a finite dimensional vector space, so is closed.  Because it is preimage of a closed set by a continuous map, $\text{Im} (D \phigpi)$ is closed.
(an equivalent argument is to note that $\L^{*}=\text{Im} F\oplus G$ for a finite dimensional subspace $G$ so $ \text{Im} D \phigpi =\text{Im} F\oplus (G\cap  \text{Im} D \phigpi$).

 {This ends the proof of the manifold structure of $\contrainte$, as a smooth submanifold of $\F$. In fact, all no KID's fibers of $\phibf$ are smooth submanifolds of $\F$. $\cqfd$}
\appendix
\section{A note for the scalar curvature case }
The manifold structure on the fiber of the scalar curvature operator is not obtained directly using $\phio(g,0)$, because the constant $\Lambda$ need  to be replaced by a function $f$. Instead, one may consider the map
$$\phi(g)=(R(g)-2f)\sqrt{g},$$
whose linearisation is
$$
D\phi(g)h= (\n^{i} \n^{j} h_{ij} - \laplag  {\tr}_{g}h) \sqrt{g} - h_{ij} \lbrack R^{ij} - \tfrac{1}{2}  (R(g) - 2 f) g^{ij} \rbrack \sqrt{g},
$$
and the adjoint  given by
$$
D\phi(g)^*N=\lbrack \n^{i} \n^{j} N - g^{ij}\laplag N - \lbrack R^{ij} - \tfrac{1}{2}  (R(g) - 2 f) g^{ij} \rbrack N \rbrack \sqrt{g}.
$$
Like for the operator $T$ defined before, the kernel of $D\phi(g)^*$ is the same than the one of
$$
N\mapsto \nabla \nabla N -\lbrack Ric(g)-\frac{1}{2(n-1)}(R(g)+2f)g\rbrack N.
$$
If $R(g)=2f$ then $D\phi(g)^*=DR(g)^*$  thus the theorem \ref{mainthmmanifoldintroscal} is obtained with $\epsilon=2f$ and the preimage of $0$ by $\phi$.

\section{About non compact manifolds with special ends}\label{noncomp}
As already explained in the introduction, the paper is written in the spirit to an easy adaptation to some non compact setting such as the asymptotically Euclidian one or the asymptotically hyperbolic context.
In such a case, the constraint operator $\phibf$  is studied for Riemannian metrics of the form $g= \gz + h$ with  $g$  ``asymptotic'' to a smooth model metric $\gz \, , \; \ie \norm{g- \gz}{\gz} = \norm{h}{\gz}$ is controlled in a suitable weighted space .\\
In this section, we explain the choices made before and mention  two  natural operators  related to the no KID's condition.

We consider a smooth  metric $\gz$ on $\m$ as model.  One can think of $\gz$ has a metric of constant sectional curvature $\kk$ on any end but this is a particular case. One  work in some weighted  Sobolev spaces, say $
\sob{k}{2}{weight} $, the weight is a real describing the asymptotic behaviour, it can change from line to line from now on.
We  usually ask $\gz$ to satisfy the following
\begin{equation}
Riem \, \gz-\frac{\kk}2 \gz \owedge \gz \in \sob{k}{2}{weight} \label{integriem},
\end{equation}
where $(\gz \owedge \gz)_{ijkl}=-2(\gz_{il} \gz_{jk} - \gz_{ik} \gz_{jl})$,
or the weaker one,
\begin{equation}
Ric \, \gz-\kk (n-1) \gz \in \sob{k}{2}{weight} \label{integriccg},
\end{equation}
or only
\begin{equation}\label{integscal}
R(\gz) -  n(n-1)\kk\in \sob{k}{2}{weight}.
\end{equation}

We fix a real parameter $\tau$ and we  set
\begin{equation}\label{defKz}
\Kz = \tau \gz \,.
\end{equation}
The cosmological constant $\Lambda$ is normalized here in dimension $n$ by
\begin{equation}\label{cosmocg}
2 \Lambda = n (n-1)(\tau^{2}+\kappa),
\end{equation}
so that $\phibf(\gz,\Kz)=0$ if $R(\gz)=\kk n(n-1)$

From the choice of $\Kz$, the conjugate momentum $\piz$ is then
\begin{equation}\label{defpiz}
\piz^{ij} =  (\Kz^{ij} -  {\tr}_{\gz}\Kz \gz^{ij}) \, d\mu(\gz) = \tau (1-n)\gz^{ij}\, d\mu(\gz) .
\end{equation}
Note that we have  $\nz \piz = \nz \Kz = 0$. If $\Ric(\mathring g)=\kk(n-1)\mathring g$ then 
$$
\mathring \Pi=-\frac12(n-1)(n-4)\tau^2\gz^{-1},
$$
$$
\mathring E=\frac12(n-1)(2k+n\tau^2)\gz^{-1}
$$
$$
\mathring \Pi-\mathring E=-(n-1)[(n-2)\tau^2+\kk]\gz^{-1}.
$$
If $h$ and $p$ are like in (\ref{defhp}) with $Y=0$ we see that
$$
F_0(y,0)=2(1-n)(\mathring\Delta y+\kk n y), 
$$
and if $h$ and $p$ are like in (\ref{defhp}) with $y=0$ we see that
$$
F_i(0,Y)=2(\mathring\Delta Y_i+\kk(n-1) Y_i).
$$
This explain the choice of the operators $A$ and $B$ of section \ref{SectionStructure}

In the non compact setting, the no KID's condition is usually not assumed but has to  be proved for a certain range of weight.
In the course, some asymptotic inequalities are useful, we introduce two operators  who have to be studied.

\noindent {\bf The operator $\mathring{\mathcal U}$}\label{sectionU}\\
The operator $\Uz$, acting on 1-forms, inspired by the formula (\ref{b29}) and so related to the covariant derivatives of Killing operator $S$. It may be used  to estimate the $W^{k+2,2}_{weight}$-norm of a 1-form $X$ with the $W^{k,2}_{weight}$-norms of $\Sz(X)$ and $(X)$ , so with the $W^{k+1,2}_{weight}$-norm of $\Sz(X)$. 
Because of the constant sectional  curvature model, we introduce
the operator  { $\mathring{\mathcal U}$}  defined on 1-forms by
\begin{equation}\label{defUz}
\mathring{\mathcal U}_{kji}(X) = \nz^{2}_{kj}X_{i} +\kk( \gz_{jk} X_{i} - \gz_{ik} X_{j}).
\end{equation}

An important step will be to prove that for any smooth one form $X$ supported on an end,
$$
||X||_{2, weight}\leq C(||\mathring{\mathcal U}(X)||_{0,weight}+||\Sz(X)||_{0,weight})
$$

\noindent {\bf The operator $\mathring{\mathcal T}$}\label{sectionT}\\
Inspired by the case $Ric(\gz)=\kk(n-1)$ and the operator $\Tz$, we define  the Obata type operator 
$$
\mathring{\mathcal T}(N)=\nz\nz N+\kk\gz N
$$
In  the same way, it has to be proven that  for
any smooth function $N$ supported on an end,
$$
||N||_{2, weight}\leq C ||\mathring{\mathcal T}(N)||_{0,weight}
$$

\section{Fredholm properties for elliptic operators with rough coefficients}
We start with a result who can be found in \cite{BH2017} or \cite{HNT2009} for instance (see also  \cite{Muller2017} and \cite{Muller2017erratum})
\begin{thm}\label{thSF}
Let $k\in\N$ such that $k+2>\frac n 2$ and $g\in H^{k+2}$. Let $P=\Delta+c$ with $c\in H^k$.
For all $j\in (-k,k+2]$ there exist a constant $C$  such that   all
$u\in H^{j}$  satisfies  the estimate
$$
\Norm{u}{j,2}\leq C (\Norm{Pu}{j-2,2}+\Norm{u}{j-2,2}).
$$
In particular $P:H^j\longrightarrow H^{j-2}$ is semi-Fredholm (with finite dimensional kernel and closed range).
\end{thm}
We are now ready to deduce (see  also \cite{Muller2017} with \cite{Muller2017erratum})
\begin{cor}\label{corFredholm}
Let $k\in\N$ such that $k+2>\frac n 2$ and $g\in H^{k+2}$. Let $P=\Delta+c$ with $c\in H^k$.
For all $j\in (-k,k+2]$ the operator $P:H^j\longrightarrow H^{j-2}$ is Fredholm.
\end{cor}
\textbf{Proof}:
We first recall that for any  $j\in[-k,k+2]$ the operator $P$ is well defined and bounded from 
$H^{j}$ to $H^{j-2}$(  \cite{BH2017}, Theorem A.1 with  $p=p_1=p_2=2$, $s_1=k+2$ and $s=s_2=j$).
From the theorem \ref{thSF},  if moreover $j\neq -k$, $P$ is semi-Fredholm so its (eventually infinite) index is well defined.

The metric $g$ and the 0-order term $c$   can be approximated by smooth ones to produce a family of operators $P_\epsilon$ with smooth coefficients  and such that
$$
\|P-P_\epsilon\|\leq \epsilon,
$$
where the norm is the norm of operators from $H^{j}$ to $H^{j-2}$.
We have  the usual elliptic regularity estimate 
$$
\Norm{v}{j,2}\leq C (\Norm{P_\epsilon v}{j-2,2}+\Norm{v}{j-2,2}),
$$
in particular $P_\epsilon$ is semi-Fredholm with finite dimensional kernel. Its formal $L^2$ adjoint (for the measure $d\mu_{\mathring g}$ for instance) $P_\epsilon^*:H^{2-j}\longrightarrow H^{-j}$ has a similar structure, verify the same kind of estimate so is also semi-Fredholm with finite dimensional kernel.
We easily deduce that $P_\epsilon$ is  a Fredhom operator.
For $\epsilon$ small enough, the  (finite) index of  $P_\epsilon$ is equal to that of $P$ (see \cite{Kato1958}) so  the index of $P$ is finite thus $P$ is Fredholm.

$\cqfd$

\begin {remark}\em
We have just used that a  semi-Fredholm limit of Fredholm  operators is Fredholm. It is clear that the proof can be transposed  to  more general operators like those in  definition
\ref{defopasympt}.
\end{remark}
\bibliography{biblioarticle}
\bibliographystyle{abbrv}

\end{document}